\documentclass[11pt]{article}
\usepackage{amsmath,amsfonts,amssymb,latexsym,color,amsthm,epsfig}
\usepackage[T1]{fontenc}
\usepackage{bbm}
\usepackage{color}
\usepackage{todonotes}       
\usepackage{mathrsfs}

\newtheorem{theorem}{Theorem}[section]
\newtheorem{prop}[theorem]{Proposition}
\newtheorem{lemma}[theorem]{Lemma}
\newtheorem{corr}[theorem]{Corollary}

\newtheorem{remark}{Remark}[section]
\newcommand{\E}{{\mathbb E}}

\newcommand{\Z}{{\mathbb Z}}
\newcommand{\PP}{\mathbb P}

\newcommand{\sss}{\scriptscriptstyle}
\newcommand{\Totm}{T^{\sss(1,2)}_{\sss{R_m}}}
\newcommand{\Totan}{T^{\sss(1,2)}_{\sss{R_{a_n}}}}
\newcommand{\Toan}{T^{\sss(1)}_{a_n}}

\newcommand{\Ntk}{\bar{N}_{\sss \mathcal{W}_n}^{(t,k)}}
\newcommand{\vep}{\varepsilon}
\newcommand{\SWG}{\mbox{SWG}}

\newcommand{\whp}{whp }
\newcommand{\Nlos}{N_{\sss \rm{los}}}
\newcommand{\NLna}{N_{\sss \mathcal{L}_n}^*}
\newcommand{\NLnaa}{N_{\sss \mathcal{L}_n}^{**}}
\newcommand{\Nlaa}{N_{\sss {\rm{los}}}^{**}}

\newcommand{\e}{{\mathrm e}}

\newcommand{\prob}{\mathbb{P}}
\newcommand{\expec}{\mathbb{E}}

\newcommand\1{\mathbbm{1}}
\newcommand{\indic}[1]{\1_{\{#1\}}}

\newcommand{\eqn}[1]{\begin{equation} #1 \end{equation}}
\newcommand{\eqan}[1]{\begin{align} #1 \end{align}}
\newcommand{\nn}{\nonumber}

\newcommand{\convd}{\stackrel{d}{\longrightarrow}}
\newcommand{\convp}{\stackrel{\PP}{\longrightarrow}}

\newcommand{\pstar}{p^{\sss \star}}

\definecolor{darkgreen}{rgb}{0,.4,0}
\definecolor{darkagenta}{rgb}{.5,0,.5}
\definecolor{darkred}{rgb}{1,0,0}
\definecolor{darkblue}{rgb}{0,0,.4}
\definecolor{black}{rgb}{0,0,0}

\newcommand{\CMnD}{\mathrm{CM}_n(\boldsymbol{D})}

\newcommand{\Goodn}{{\sf Good}_n}

\newcommand{\mulos}{\mu^{\sss (\mathcal{L})}}
\newcommand{\muwin}{\mu^{\sss (\mathcal{W})}}
\newcommand{\Vlos}{V^{\sss (\mathcal{L})}}
\newcommand{\Vwin}{V^{\sss (\mathcal{W})}}
\newcommand{\Elos}{E^{\sss (\mathcal{L})}}
\newcommand{\Ewin}{E^{\sss (\mathcal{W})}}
\newcommand{\Slos}{S^{\sss (\mathcal{L})}}
\newcommand{\Swin}{S^{\sss (\mathcal{W})}}
\newcommand{\Ltwin}{L_{\sss \mathcal{W}_n}^{\sss(t)}}
\newcommand{\barLtwin}{\bar{L}_{\sss \mathcal{W}_n}^{\sss(t)}}
\newcommand{\Ntkwin}{N_{\sss \mathcal{W}_n}^{\sss(k+1,t)}}
\newcommand{\barNtkwin}{\bar{N}_{\sss \mathcal{W}_n}^{\sss(k+1,t)}}
\newcommand{\field}{\mathscr{F}}
\newcommand{\op}{o_{\sss \prob}}

\begin{document}
\parskip=5pt plus1pt minus1pt \parindent=15pt
\title{The winner takes it all}
\author{Maria Deijfen\thanks{Department of Mathematics, Stockholm University, 106 91 Stockholm, Sweden; {\tt mia@math.su.se}} \and Remco van der Hofstad\thanks{Department of Mathematics and Computer Science, Eindhoven University of Technology, Box 513, 5600 MB Eindhoven, The Netherlands; {\tt rhofstad@win.tue.nl}}}
\date{January 2016}
\maketitle

\begin{abstract}
\noindent We study competing first passage percolation on graphs generated by the configuration model. At time 0, vertex 1 and vertex 2 are infected with the type 1 and the type 2 infection, respectively, and an uninfected vertex then becomes type 1 (2) infected at rate $\lambda_1$ ($\lambda_2$) times the number of edges connecting it to a type 1 (2) infected neighbor. Our main result is that, if the degree distribution is a power-law with exponent $\tau\in(2,3)$, then, as the number of vertices tends to infinity and with high probability, one of the infection types will occupy all but a finite number of vertices. Furthermore, which one of the infections wins is random and both infections have a positive probability of winning regardless of the values of $\lambda_1$ and $\lambda_2$. The picture is similar with multiple starting points for the infections.

\vspace{0.5cm}


\vspace{0.5cm}

\end{abstract}

\section{Introduction}


Consider a graph generated by the configuration model with random independent and identically distributed (i.i.d.) degrees, that is, given a finite number $n$ of vertices, each vertex is independently assigned a random number of half-edges according to a given probability distribution and the half-edges are then paired randomly to form edges (see below for more details). Independently assign two exponentially distributed passage times $X_1(e)$ and $X_2(e)$ to each edge $e$ in the graph, where $X_1(e)$ has parameter $\lambda_1$ and $X_2(e)$ parameter $\lambda_2$, and let two infections controlled by these passage times compete for space on the graph. More precisely, at time 0, vertex 1 is infected with the type 1 infection, vertex 2 is infected with the type 2 infection and all other vertices are uninfected. The infections then spread via nearest neighbors in the graph in that the time that it takes for the type 1 (2) infection to traverse an edge $e$ and invade the vertex at the other end is given by $X_1(e)$ ($X_2(e)$). Furthermore, once a vertex becomes type 1 (2) infected, it stays type 1 (2) infected forever and it also becomes immune to the type 2 (1) infection. Note that, since the vertices are exchangeable in the configuration model, the process is equivalent in distribution to the process obtained by infecting two randomly chosen vertices at time 0.

We shall impose a condition on the degree distribution that guarantees that the underlying graph has a giant component that comprises almost all vertices. According to the above dynamics, almost all vertices will then eventually be infected. We are interested in asymptotic properties of the process as $n\to\infty$. Specifically, we are interested in comparing the fraction of vertices occupied by the type 1 and the type 2 infections, respectively, when the degree distribution is a power law with exponent $\tau\in(2,3)$, that is, when the degree distribution has finite mean but infinite variance. Our main result is roughly that the probability that {\em both} infection types occupy positive fractions of the vertex set is 0 for all choices of $\lambda_1$ and $\lambda_2$. Moreover, the winning type will in fact conquer all but a finite number of vertices. A natural guess is that asymptotic coexistence is possible if and only if the infections have the same intensity -- which for instance is the case for first passage percolation on $\mathbb{Z}^d$ and on random regular graphs; see Section \ref{sec:related} -- but this is hence not the case in our setting.

\subsection{The configuration model}\label{sec:conf}

Let $[n]\equiv \{1,2,\ldots, n\}$ denote the vertex set of the graph and $D_1,\ldots, D_n$ the degrees of the vertices. The degrees are i.i.d.\ random variables, and we shall throughout assume that
\begin{itemize}
\item[(A1)] $\PP(D\geq 2)=1$;
\item[(A2)] there exists a $\tau\in(2,3)$ and constants $c_2\geq c_1>0$ such that, for all $x>0$,
	\eqn{
	c_1 x^{-(\tau-1)}\leq \PP(D>x)\leq c_2 x^{-(\tau-1)}.
	}
\end{itemize}

For some results, the assumption (A2) will be strengthened to

\begin{itemize}
 \item [(A2')] there exist $\tau\in(2,3)$ and $c_{\sss D}\in (0,\infty)$ such that $\PP(D>x)=c_{\sss D} x^{-(\tau-1)}(1+o(1))$.
\end{itemize}

As described above, the graph is constructed in that each vertex $i$ is assigned $D_i$ half-edges, and the half-edges are then paired randomly: first we pick two half-edges at random and create an edge out of them, then we pick two half-edges at random from the set of remaining half-edges and pair them into an edge, etc. If the total degree happens to be odd, then we add one half-edge at vertex $n$ (clearly this will not affect the asymptotic properties of the model). The construction can give rise to self-loops and multiple edges between vertices, but these imperfections will be relatively rare when $n$ is large; see \cite{Remco_notes, Jans06b}.

It is well-known that the critical point for the occurrence of a giant component -- that is, a component comprising a positive fraction of the vertices as $n\to\infty$ -- in the configuration model is given by $\nu:=\E[D(D-1)]/\E[D]=1$; see e.g. \cite{SvanteMalwina,MR1,MR2}. The quantity $\nu$ is the reproduction mean in a branching process with offspring distribution $D^\star-1$ where $D^\star$ is a size-biased version of the degree variable. More precisely, with $(p_d)_{d\geq 1}$ denoting the degree distribution, the offspring distribution is given by
	\begin{equation}\label{eq:forward_degree}
	\pstar_d=\frac{(d+1)p_{d+1}}{\E[D]}.
	\end{equation}
Such a branching process approximates the initial stages of the exploration of the components in the configuration model, and the asymptotic relative size of the largest component in the graph is given by the survival probability of the branching process \cite{SvanteMalwina,MR1,MR2}. When the degree distribution is a power-law with exponent $\tau\in(2,3)$, as stipulated in (A2), it is easy to see that $\nu=\infty$ so that the graph is always supercritical. Moreover, the assumption (A1) implies that the survival probability of the branching process is 1 so that the asymptotic fraction of vertices in the giant component converges to 1.

\subsection{Main result}\label{sec:main}

Consider two infections spreading on a realization of the configuration model according to the dynamics described in the beginning of the section, that is, an uninfected vertex becomes type 1 (2) infected at rate $\lambda_1$ ($\lambda_2$) times the number of edges connecting it to type 1 (2) infected neighbors. Note that, by time-scaling, we may assume that $\lambda_1=1$ and write $\lambda_2=\lambda$. Let $N_i(n)$ denote the final number of type $i$ infected vertices, and write $\bar{N}_i(n)=N_i(n)/n$ for the final fraction of type $i$ infected vertices. As mentioned, the assumption (A2) guarantees that almost all vertices in the graph form a single giant component. Hence $\bar{N}_1(n)+\bar{N}_2(n)\convp 1$ and it is therefore sufficient to consider $\bar{N}_1(n)$. Define
$N_{\sss \rm{los}}(n)=\min\{N_1(n),N_2(n)\}$ so that $N_{\sss \rm{los}}(n)$ is the total number of vertices captured by the losing type, that is, the type that occupies the smallest number of vertices. The following is our main result:

\begin{theorem}[The winner takes it all]
\label{th:main} Fix $\lambda$ and write $\mu=1/\lambda$.
\begin{itemize}
\item[\rm{(a)}] The fraction $\bar{N}_1(n)$ of type 1 infected vertices converges in distribution to the indicator variable $\indic{V_1<\mu V_2}$ as $n\to\infty$, where $V_1$ and $V_2$ are i.i.d.\ proper random variables with support on $\mathbb{R}^+$.
\item[\rm{(b)}] Assume (A2'). The total number $N_{\sss \rm{los}}(n)$ of vertices occupied by the losing type converges in distribution to a proper random variable $N_{\sss \rm{los}}$.
\end{itemize}
\end{theorem}

\begin{remark}[Explosion times]
{\rm The variables $V_i$ ($i=1,2$) are distributed as explosion times of a certain continuous-time branching process with infinite mean. The process is started from $D_i$ individuals, representing the edges of vertex $i$, and will be characterized in more detail in Section \ref{sec:prel}. In part (b), the limiting random variable $N_{\sss \rm{los}}$ has an explicit characterization involving the (almost surely finite) extinction time of a certain Markov process; see Section \ref{sec:main_b}. In fact, the proof reveals that the limiting number of vertices that is captured by the losing type is equal to 1 with strictly positive probability, which is the smallest possible value. Thus, the ABBA lyrics `The winner takes it all. The loser's standing small...' could not be more appropriate.}
\end{remark}

Roughly stated, the theorem implies that coexistence between the infection types is never possible. Instead, one of the infection types will invade all but a finite number of vertices and, regardless of the relation between the intensities, both infections have a positive probability of winning. The proof is mainly based on ingredients from \cite{RGSc}, where standard first passage percolation (that is, first passage percolation with one infection type and exponential passage times) on the configuration model is analyzed.

Let us first give a short heuristic explanation. Here and throughout the paper, a sequence of events is said to occur with high probability (whp) when their probabilities tend to 1 as $n\to\infty$. Whp, the initially infected vertex 1 and vertex 2 will not be located very close to each other in the graph and hence the infection types will initially evolve without interfering with each other. This means that the initial stages of the spread of each one of the infections can be approximated by a continuous-time branching process, which has infinite mean when the degree distribution has infinite variance (because of size biasing). These two processes will both explode in finite time, and the type that explodes \emph{first} is random and asymptotically equal to 1 precisely when $V_1<\mu V_2$. Theorem \ref{th:main} follows from the fact that the type with the smallest explosion time will get a lead that is impossible to catch up with for the other type. More specifically, the type that explodes first will \whp occupy all vertices of high degree -- often referred to as {\em hubs} -- in the graph shortly after the time of explosion, while the other type occupies only a finite number of vertices. From the hubs the exploding type will then rapidly invade the rest of the graph before the other type makes any substantial progress at all.

We next investigate the setting where we start the competition from several vertices chosen uniformly at random:

\begin{theorem}[Multiple starting points]
\label{th:main-multiple}
Fix $\lambda$ and write $\mu=1/\lambda$. Also fix integers $k_1,k_2\geq 1$, and start with $k_1$ type 1 infected vertices and $k_2$ type 2 infected vertices chosen uniformly at random from the vertex set.
\begin{itemize}
\item[\rm{(a)}] The fraction $\bar{N}_1(n)$ of type 1 infected vertices converges in distribution to the indicator variable $\indic{V_{1,k_1}<\mu V_{2,k_2}}$ as $n\to\infty$, where $V_{1,k_1}$ and $V_{2,k_2}$ are two independent proper random variables with support on $\mathbb{R}^+$.
\item[\rm{(b)}] Assume (A2'). The total number $N_{\sss \rm{los}}(n)$ of vertices occupied by the losing type converges in distribution to a proper random variable $N_{\sss \rm{los}}$.
\item[\rm{(c)}] Assume (A2'). For every $k_1,k_2\geq 1$, it holds that $\prob(V_{1,k_1}<\mu V_{2,k_2})\in(0,1)$. Moreover, for fixed $\alpha\in (0,\infty)$, as $k\rightarrow \infty$,
	\eqn{
	\label{asympt-prob-win}
	\prob(V_{1,k}<\mu V_{2,\alpha k})
	\rightarrow  \prob(Y_1<\mu \alpha^{3-\tau} Y_2)\in(0,1),
	}
where $Y_1,Y_2$ are two i.i.d.\ random variables with distribution
	$$
	Y=\int_0^{\infty} \frac{1}{1+Q_t}dt,
	$$
for a stable subordinator $(Q_t)_{t\geq 0}$ with $\expec[\e^{-s Q_t}]=\e^{-\sigma s^{\tau-2}t}$ for some $\sigma=\sigma(c_{\sss D})$.
\end{itemize}
\end{theorem}

\begin{remark}[Explosion times revisited]
{\rm The variable $V_{i,k_i}$ has the distribution of the explosion time of a continuous-time branching process with the same reproduction rules as in the case with a single initial type $i$ vertex, but now the number of individuals that the process is started from is distributed as $D_1+\cdots+D_{k_i}$ and represents the total degree of the $k_i$ initial type $i$ vertices. The scaling of the explosion time of the branching process started from $k$ individuals for large $k$ is investigated in more detail in Lemma \ref{lem-V(k)-asymp}.}
\end{remark}

In Theorem \ref{th:main-multiple}, we see that the fastest species does not necessarily win even when it has twice as many starting points, but it does when $\alpha\rightarrow \infty$, that is, when starting from a much larger number of vertices than the slower species. We only prove Theorem \ref{th:main-multiple} in the case where $k_1=k_2=1$, in which case it reduces to Theorem \ref{th:main}. The case where $(k_1,k_2)\neq (1,1)$ is similar. Hence only the proof of \eqref{asympt-prob-win} in Theorem \ref{th:main-multiple}(c) is provided in detail; see Section \ref{sec:main_b}.

\subsection{Related work and open problems}\label{sec:related}

First passage percolation on various types of discrete probabilistic structures has been extensively studied; see e.g.\ \cite{RGSer,BHH12, GrimKest,HamWel,LyonsPem,SW}. The classical example is when the underlying structure is taken to be the $\mathbb{Z}^d$-lattice. The case with exponential passage times is then often referred to as the Richardson model and the main focus of study is the growth and shape of the infected region \cite{CoxDur,Kesten_speed,NewPiz,Richardson}. The Richardson model has also been extended to a two-type version that describes a competition between two infection types; see \cite{2typeR}. Infinite coexistence then refers to the event that both infection types occupy infinite parts of the lattice, and it is conjectured that this has positive probability if and only if the infections have the same intensity. The if-direction was proved for $d=2$ in \cite{2typeR} and for general $d$ independently in \cite{GM} and \cite{Hoff_coex}. The only-if-direction remains unproved, but convincing partial results can be found in \cite{aa}. On $\Z^d$, the starting points of the competing species are typically taken to be two neighboring vertices. Doing this in our setup on the configuration model would in principle not change our main results. Specifically, letting the species start at either end of a randomly chosen edge would change the limiting probabilities of winning for the species, but the fact that one wins and the other occupies a bounded number of vertices would remain unchanged.

As for the configuration model, the area of network modeling has been very active the last decade and the configuration model is one of the most studied models. One of its main advantages is that it gives control over the degree distribution, which is an important quantity in a network with great impact on global properties. As mentioned, first passage percolation with exponential edge weights on the configuration model has been analyzed in \cite{RGSc}. The results there revolve around the length of the time-minimizing path between two vertices and the time that it takes to travel along such a path. In \cite{BHH12}, these results are extended to all continuous edge-weight distributions under the assumption of finite variance degrees.

Recently, in \cite{regular}, competing first passage percolation has been studied on so-called random regular graphs, which can be generated by the configuration model with constant degree, that is, with $\PP(D=d)=1$ for some $d$. The setup in \cite{regular} allows for a number of different types of starting configurations, and the main result relates the asymptotic fractions occupied by the respective infection types to the sizes of the initial sets and the intensities. When the infections are started from two randomly chosen vertices, coexistence occurs with probability 1 if the infections have the same intensity, while, when one infection is stronger than the other, the stronger type wins, as one might expect. The somewhat counterintuitive result in the present paper is hence a consequence of large variability in the degrees. We conjecture that the result formulated here remains valid precisely when the explosion time of the corresponding continuous-time branching process is finite. See \cite{Grey74} for a discussion of explosion times for age-dependent branching processes.

A natural continuation of the present work is to study the case when $\tau>3$, that is, when the degree distribution has finite variance. We conjecture that the result is then the same as for constant degrees as described above. Another natural extension is to investigate other types of distributions for the passage times. The results may then well differ from the exponential case. For instance, ongoing work on the case with constant passage times (possibly different for the two species) and $\tau\in(2,3)$ indicates that the fastest species always wins, but that there can be coexistence when the passage times are equal \cite{comp_const, comp_const2}.

Finally we mention the possibility of investigating whether the results generalize to other graph structures with similar degree distribution, e.g.\ inhomogeneous random graphs and graphs generated by preferential attachment mechanisms, see \cite{prefatt} for results on preferential attachment networks.

\section{Preliminaries}\label{sec:prel}

In this section we summarize the results on one-type first passage percolation from \cite{RGSc} that we shall need. Theorem \ref{th:main}(a) and \ref{th:main}(b) are then proved in Section \ref{sec:main_a} and \ref{sec:main_b}, respectively. Also, the proof of the asymptotic characterization (\ref{asympt-prob-win}) is given in Section \ref{sec:main_b}.

Let each edge in a realization of the configuration model independently be equipped with {\em one} exponential passage time with mean 1. In summary, it is shown in \cite{RGSc} that, when the degree distribution satisfies (A1) and (A2), the asymptotic minimal time between vertex 1 and vertex 2 is given by $V_1+V_2$, where $V_1$ and $V_2$ are i.i.d.\ random variables indicating the explosion time of an infinite mean continuous-time branching process that approximates the initial stages of the flow through the graph starting from vertex 1 and 2 respectively; see below. The result follows roughly by showing that the sets of vertices that can be reached from vertex 1 and 2, respectively, within time $t$ are \whp disjoint up until the time when the associated branching processes explode, and that they then hook up, creating a path between 1 and 2.

\paragraph{Exploration of first-passage percolation on the configuration model.}
To be a bit more precise, we first describe a natural stepwise procedure for exploring the graph and the flow of infection through it starting from a given vertex $v$. Let ${\rm SWG}_m^{\sss (v)}$ denote the graph consisting of the set of explored vertices and edges after $m$ steps, where SWG stands for Smallest-Weight Graph. Write $\mathcal{U}_m^{\sss (v)}$ for the set of unexplored half-edges emanating from vertices in ${\rm SWG}_m^{\sss (v)}$ and define $S_m^{\sss (v)}:=|\mathcal{U}_m^{\sss (v)}|$. Finally, let $\mathcal{F}_m^{\sss (v)}$ denote the set of half-edges belonging to vertices in the complement of ${\rm SWG}_m^{\sss (v)}$. When there is no risk of confusion, we will often omit the superscript $v$ in the notation. Set ${\rm SWG}_1=\{v\}$, so that $S_1=D_v$. Given ${\rm SWG}_m$, the graph ${\rm SWG}_{m+1}$ is constructed as follows:

\begin{itemize}
\item[1.] Pick a half-edge at random from the set $\mathcal{U}_m$. Write $x$ for the vertex that this half-edge is attached to, and note that $x\in {\rm SWG}_m$.
\item[2.] Pick another half-edge at random from $\mathcal{U}_m\cup\mathcal{F}_m$ and write $y$ for the vertex that this half-edge is attached to.
\item[3.] If $y\not\in {\rm SWG}_m$ -- that is, if the second half-edge is in $\mathcal{F}_m$ -- then ${\rm SWG}_{m+1}$ consists of ${\rm SWG}_m$ along with the vertex $y$ and the edge $(x,y)$. If $n$ is large and $m$ is much smaller than $n$, then this is the most likely scenario.
\item[4.] If $y\in {\rm SWG}_m$ -- that is, if the second half-edge is in $\mathcal{U}_m$ -- then ${\rm SWG}_{m+1}={\rm SWG}_m$ and the two selected half-edges are removed from the exploration process. This means that we have detected a cycle in the graph, and that the corresponding edge will not be used to transfer the infection.
\end{itemize}

The above procedure can be seen as a discrete-time representation of the flow through the graph observed at the times when the infection traverses a new edge: Each unexplored half-edge emanating from a vertex that has already been reached by the flow has an exponential passage time with mean 1 attached to it. In step 1, we pick such a half-edge at random, which is equivalent to picking the one with the smallest passage time. In step 2, we check where the chosen half-edge is connected. When this vertex has not yet been reached by the flow, it is added to the explored graph along with the connecting edge in step 3. When the vertex has already been reached by the flow, only the edges is added in step 4, thus creating a cycle.

As for the number of unexplored half-edges emanating from explored vertices, this is increased by the forward degree of the added vertex minus 1 in case a vertex is added, and decreased by 2 in case a cycle is detected. Hence, defining
	$$
	B_i=\left\{ \begin{array}{ll}
                      \mbox{the forward degree of the added vertex if a vertex is added in step $i$};\\
                      -1 \mbox{ if a cycle is created in step $i$},
                    \end{array}
            \right.
	$$
we have for $m\geq 2$ that
	$$
	S_m=D_v+\sum_{i=2}^m(B_i-1).
	$$
Denote the total time of the first $m$ steps by $T_m$ and let $(E_i)_{i=1}^\infty$ be a sequence of i.i.d.\ Exp(1)-variables. The time for traversing the edge that is explored in the $i$th step is the minimum of $S_i$ i.i.d.\ exponential variables with mean 1 and thus it has the same distribution as $E_i/S_i$. Hence
	\begin{equation}\label{eq:Tm_distr}
	T_m\stackrel{d}{=}\sum_{i=1}^m\frac{E_i}{S_i}.
	\end{equation}
Write $\mathcal{V}(G)$ for the vertex set of a graph $G$, let $|\mathcal{V}(G)|$ denote its size, and define
	\begin{equation}\label{eq:Rm}
	R_m=\inf\{j\colon |\mathcal{V}({\rm SWG}_j)|\geq m\},
	\end{equation}
that is, $R_m$ is the step when the $m$th vertex is added to the explored graph. Since no vertex is added in a step where a cycle is created, we have that $R_m\geq m$. However, if $n$ is large and $m$ is small in relation to $n$, it is unlikely to encounter cycles in the early stages of the exploration process and thus $R_m\approx m$ for small $m$. Hence, we should be able to replace $m$ by $R_m$ above and still obtain quantities with similar behavior. Indeed, Proposition \ref{prop:prel} below states that $T_{R_m}$ (the time until the flow has reached $m$ vertices) and $T_m$ have the same limiting distribution as $n\to\infty$ as long as $m=m_n$ is not too large.

\paragraph{Passage times for smallest-weight paths.}
To identify the limiting distribution of $T_m$, note that, as long as no cycles are encountered, the exploration graph is a tree and its evolution can therefore be approximated by a continuous-time branching process. The root is the starting vertex $v$, which dies immediately and leaves behind $D_v$ children, corresponding to the $D_v$ half-edges incident to $v$. All individuals (=unexplored half-edges) then live for an Exp(1)-distributed amount of time, independently of each other, and when the $i$th individual dies it leaves behind $\widetilde{B}_i$ children, where $(\widetilde{B}_i)_{i\geq 1}$ is an i.i.d.\ sequence with distribution (\ref{eq:forward_degree}). Indeed, as long as no cycles are created, the offspring of a given individual is the forward degree of the corresponding vertex, and the forward degrees of explored vertices are asymptotically independent with the size-biased distribution specified in (\ref{eq:forward_degree}). The number of alive individuals after $m\geq 2$ steps in the approximating branching process, corresponding to the number of unexplored half-edges incident to the graph at that time, is given by
	$$
	\widetilde{S}_m=D_v+\sum_{i=2}^m(\widetilde{B}_i-1)
	$$
and hence the time when the total offspring reaches size $m$ is equal in distribution to $\sum_{i=1}^mE_i/\widetilde{S}_i$. In \cite{RGSc} it is shown that the branching process approximation remains valid for $m=m_n\to\infty$ as long as $m_n$ does not grow too fast with $n$. Define
	\eqn{
	\label{an-choice}
	a_n=n^{(\tau-2)/(\tau-1)}.
	}
It turns out that ``does not grow too fast'' means roughly that $m_n=o(a_n)$. The intuition behind the choice of $a_n$ is that for $\tau\in(2,3)$, there is a large discrepancy between the number of alive and the number of dead individuals. In particular, $a_n$ in \eqref{an-choice} equals the asymptotic number of dead individuals in each of two SWGs emanating from vertex 1 and 2, respectively, at the moment when the two SWGs collide. This is explained in more detail in \cite[(4.21-4.25)]{RGSc}. To summarize, in the above comparison of the SWG to a branching process, we see that we grow the graph (in terms of the pairing of the half-edges) simultaneously with the exploration of the neighborhood structure in the graph, which is approximated by a (continuous-time) branching process.

Write $X(u\leftrightarrow v)$ for the passage time between the vertices $u$ and $v$, that is, $X(u\leftrightarrow v)=T_{m(u,v)}$ with $m(u,v)=\inf\{m:v\in{\rm SWG}_m^{\sss (u)}\}$. The relevant results from \cite{RGSc} are summarized in the following proposition. Here, part (a) is essential in proving part (b), part (d) follows by combining parts (b) and (c), and part (e) by combining parts (b) - (d). For details we refer to \cite{RGSc}: Part (a) is Proposition 4.7, part (b) is Proposition 4.6(b), where the characterization of $V$ is made explicit in (6.14) in the proof, part (c) is Proposition 4.9 and, finally, part (e) is Theorem 3.2(b).

\begin{prop}[Bhamidi, van der Hofstad, Hooghiemstra (2010)]\label{prop:prel}
Consider first passage percolation on a graph generated by the configuration model with a degree distribution that satisfies (A1) and (A2).

\begin{itemize}
\item[\rm{(a)}] There exists a $\rho>0$ such that the sequence $(B_i)_{i\geq 1}$ can be coupled to the i.i.d.\ sequence $(\widetilde{B}_i)_{i\geq 1}$ with law (\ref{eq:forward_degree}) in such a way that $(B_i)_{i=2}^{n^\rho}=(\widetilde{B}_i)_{i=2}^{n^\rho}$ whp.

\item[\rm{(b)}] Let $\bar{m}_n$ be such that $\log(\bar{m}_n/a_n)=o(\sqrt{\log n})$ and assume that $m=m_n\to\infty$ is such that $m_n\leq \bar{m}_n$. As $n\to\infty$, the times $T_m$ and $T_{\sss {R_m}}$ both converge in distribution to a proper random variable $V$, where
    \begin{equation}\label{V-rv-def}
    V\stackrel{d}{=}\sum_{i=1}^\infty\frac{E_i}{\widetilde{S}_i}.
    \end{equation}
The law of $V$ has the interpretation of the explosion time of the approximating branching process.

\item[\rm{(c)}] For $m=m_n=o(a_n)$ and any two fixed vertices $u$ and $v$, the two exploration graphs ${\rm SWG}^{\sss(u)}_{a_n}$ and ${\rm SWG}^{\sss(v)}_m$ are \whp disjoint. Furthermore, at time $m={\sf C}_n$, the graph ${\rm SWG}_m^{\sss (u)}\cup {\rm SWG}_m^{\sss (v)}$ becomes connected, where ${\sf C}_n/a_n$ converges in distribution to an a.s.\ finite random variable.

\item[\rm{(d)}] Let $m=m_n\to\infty$, with $m_n\leq \bar{m}_n$, and fix two vertices $u$ and $v$. Then $(T_{m_n}^{\sss(u)}, T_{m_n}^{\sss(v)})\convd (V_u,V_v)$ as $n\to\infty$, where $(V_u,V_v)$ are independent copies of the random variable in \eqref{V-rv-def}.
\item[\rm{(e)}] The passage time $X(u\leftrightarrow v)$ converges in distribution to a random variable distributed as $V_u+V_v$.
\end{itemize}
\end{prop}

\paragraph{Coupling of competition to first passage percolation.} We now return to the setting with two infection types that are imposed at time 0 at the vertices 1 and 2 and then spread at rate 1 and $\lambda$, respectively. Recall that $\mu=1/\lambda$. The following coupling of the two infection types will be used in the rest of the paper: Each edge $e=(u,v)$ is equipped with one single exponentially distributed random variable $X(e)$ with mean 1. The infections then evolve in that, if $u$ is type 1 (2) infected, then the time until the infection reaches $v$ via the edge $(u,v)$ is given by $X(u,v)$ ($\mu X(u,v)$) and, if vertex $v$ is uninfected at that point, it becomes type 1 (2) infected.

Under competition, the above exploration of the flow of infection is adjusted as follows. Let ${\rm SWG}_m^{\sss (1,2)}$ denote the graph consisting of the set of explored vertices and edges after $m$ steps. We split ${\rm SWG}_m^{\sss (1,2)}={\rm SWG}_m^{\sss (\underline{1},2)}\cup {\rm SWG}_m^{\sss (1,\underline{2})}$, where ${\rm SWG}_m^{\sss (\underline{1},2)}$ and ${\rm SWG}_m^{\sss (1,\underline{2})}$ denote the part that is occupied by type 1 and type 2, respectively. Also write $\mathcal{U}_m^{\sss (1,2)}$ for the set of unexplored half-edges emanating from vertices in ${\rm SWG}_m^{\sss (1,2)}$ and split it as $\mathcal{U}_m^{\sss (1,2)}=\mathcal{U}_m^{\sss (\underline{1},2)}\cup \mathcal{U}_m^{\sss (1,\underline{2})}$, where $\mathcal{U}_m^{\sss (\underline{1},2)}$ and $\mathcal{U}_m^{\sss (1,\underline{2})}$ denote half-edges attached to vertices infected by type 1 and type 2, respectively. Write $S_m^{\sss (\underline{1},2)}:=|\mathcal{U}_m^{\sss (\underline{1},2)}|$ and $S_m^{\sss (1,\underline{2})}:=|\mathcal{U}_m^{\sss (1,\underline{2})}|$. Finally, the set of half-edges belonging to vertices in the complement of ${\rm SWG}_m^{\sss (1,2)}$ is denoted $\mathcal{F}_m^{\sss (1,2)}$. Set ${\rm SWG}_1^{\sss (\underline{1},2)}=\{1\}$ and ${\rm SWG}_1^{\sss (1,\underline{2})}=\{2\}$. Given ${\rm SWG}_m^{\sss (\underline{1},2)}$ and ${\rm SWG}_m^{\sss (1,\underline{2})}$, the graphs ${\rm SWG}_{m+1}^{\sss (\underline{1},2)}$ and ${\rm SWG}_{m+1}^{\sss (1,\underline{2})}$ are constructed as follows:

\begin{itemize}
\item[1.] With probability $S_m^{\sss (\underline{1},2)}/(S_m^{\sss (\underline{1},2)}+\lambda S_m^{\sss (1,\underline{2})})$, pick a half-edge at random from the set $\mathcal{U}_m^{\sss (\underline{1},2)}$, and with the complementary probability, pick a half-edge at random from $\mathcal{U}_m^{\sss (1,\underline{2})}$. Write $x$ for the vertex that this half-edge is incident to.

\item[2.] Pick another half-edge at random from $\mathcal{U}_m^{\sss (1,2)}\cup\mathcal{F}_m^{\sss (1,2)}$ and write $y$ for the vertex that this half-edge is incident to.

\item[3.] If $y\not\in {\rm SWG}_m^{\sss (1,2)}$ and $x\in {\rm SWG}_m^{\sss (\underline{1},2)}$ -- that is, if $y$ is not yet explored and $x$ is type 1 infected -- then ${\rm SWG}_{m+1}^{\sss (\underline{1},2)}$ consists of ${\rm SWG}_m^{\sss (\underline{1},2)}$ along with the vertex $y$ and the edge $(x,y)$ while ${\rm SWG}_{m+1}^{\sss (1,\underline{2})}={\rm SWG}_m^{\sss (1,\underline{2})}$. Similarly, if $y\not\in {\rm SWG}_m^{\sss (1,2)}$ and $x\in {\rm SWG}_m^{\sss (1,\underline{2})}$, then ${\rm SWG}_{m+1}^{\sss (1,\underline{2})}$ consists of ${\rm SWG}_m^{\sss (1,\underline{2})}$ along with the vertex $y$ and the edge $(x,y)$ while ${\rm SWG}_{m+1}^{\sss (\underline{1},2)}={\rm SWG}_m^{\sss (\underline{1},2)}$.

\item[4.] If $y\in {\rm SWG}_m^{\sss (1,2)}$ -- that is, if $y$ is already explored -- then ${\rm SWG}_{m+1}^{\sss (1,2)}={\rm SWG}_m^{\sss (1,2)}$ and the selected half-edges are removed from the exploration process. Indeed, since both $x$ and $y$ are already infected, the edge will not be used to transfer the infection.
\end{itemize}

Note that, by Proposition \ref{prop:prel}, for $m=o(a_n)$, the graph ${\rm SWG}_m^{\sss (1,2)}$ consists whp of two disjoint components given by the SWGs obtained with one-type exploration from vertex 1 and vertex 2, respectively. In what follows, we will work both with quantities based on one-type exploration and on exploration under competition. Quantities based on a one-type process are equipped with a single superscript (e.g.\ $T_{m}^{\sss (1)}$), while quantities based on competition are equipped with double superscripts (e.g.\ $T_{m}^{\sss ({1,2})}$) (but will often be simplified).

\section{Proof of Theorem \ref{th:main}(a)}\label{sec:main_a}

In this section we prove Theorem \ref{th:main}(a). Recall that the randomness in the process is represented by one single Exp(1)-variable per edge, as described above. All random times based on the one-type exploration that appear in the sequel are based on these variables and are then multiplied by $\mu=1/\lambda$ to obtain the corresponding quantities for Exp($\lambda$)-variables. Following the notation in the previous section, we write $T_{a_n}^{\sss (i)}$ for $T_{a_n}$ when the growth is started from vertex $i$. Furthermore, for $i=1,2$, we write $V_i$ for the distributional limit as $n\to\infty$ of $T_{a_n}^{\sss (i)}$, where $V_i$ are characterized in Proposition \ref{prop:prel}(b). The main technical result is stated in the following proposition:

\begin{prop}\label{prop:type1}
Fix $\mu\leq 1$ and let $U$ be a vertex chosen uniformly at random from the vertex set. As $n\to\infty$,
	$$
	\PP\left(U \mbox{ is type 1 infected}\mid \,T_{a_n}^{\sss (1)}<\mu T_{a_n}^{\sss (2)}\right)\to 1
	$$
and
    $$
	\PP\left(U \mbox{ is type 2 infected}\mid\,T_{a_n}^{\sss (1)}>\mu T_{a_n}^{\sss (2)}\right)\to 1.
	$$
\end{prop}

With this proposition at hand, Theorem \ref{th:main}(a) follows easily:

\begin{proof}[Proof of Theorem \ref{th:main}(a)] It follows from Proposition \ref{prop:type1} that
	$$
	\E[\bar{N}_1(n)\mid T_{a_n}^{\sss (1)}<\mu T_{a_n}^{\sss (2)}]=\PP\left(U \mbox{ is type 1 infected}
	\mid T_{a_n}^{\sss (1)}<\mu T_{a_n}^{\sss (2)}\right)\to 1,
	$$
and, similarly,
	$$
	\E[\bar{N}_1(n)\mid T_{a_n}^{\sss (1)}>\mu T_{a_n}^{\sss (2)}]
	=\PP\left(U \mbox{ is type 1 infected}\mid T_{a_n}^{\sss (1)}>\mu T_{a_n}^{\sss (2)}\right)\to 0.
	$$
By the Markov inequality, this implies that
	\eqan{
	\PP(\bar{N}_1(n)<1-\vep\mid T_{a_n}^{\sss (1)}<\mu T_{a_n}^{\sss (2)})
	&=\PP(\bar{N}_2(n)>\vep \mid T_{a_n}^{\sss (1)}<\mu T_{a_n}^{\sss (2)})\\
	&\leq \frac{1}{\vep}\E[\bar{N}_2(n)\mid T_{a_n}^{\sss (1)}<\mu T_{a_n}^{\sss (2)}]\to 0,\nn
	}
so that $\PP(\bar{N}_1(n)>1-\vep\mid T_{a_n}^{\sss (1)}<\mu T_{a_n}^{\sss (2)})\to 1$ for any $\vep>0$.
Similarly, $\PP(\bar{N}_1(n)<\vep\mid T_{a_n}^{\sss (1)}>\mu T_{a_n}^{\sss (2)})\to 1$ for any $\vep>0$. Since $\bar{N}_1(n)\in[0,1]$ and $\PP(T_{a_n}^{\sss (1)}<\mu T_{a_n}^{\sss (2)})\to\PP(V_1<\mu V_2)$, Theorem \ref{th:main}(a) follows from this. 
\end{proof}
\medskip

Let $\vep_n\searrow 0$, with $\vep_n\geq c/\log\log n$ for some constant $c$, and define $A_n=\{T_{a_n}^{\sss (1)}+\varepsilon_n<\mu (T_{a_n}^{\sss (2)}-\varepsilon_n)\}$. We remark that $\vep_n=c/\log\log n$ suffices for Lemma \ref{le:3}, but that we may have to take $\vep_n$ larger when applying Lemma \ref{le:4}. In order to prove Proposition \ref{prop:type1}, we will show that
	\begin{equation}
	\label{eq:type1_given_A_n}
	\PP\big(U\mbox{ is type 1 infected}\mid A_n\big)\to 1.
	\end{equation}
With $B_n=\{T_{a_n}^{\sss (1)}-\varepsilon_n>\mu (T_{a_n}^{\sss (2)}+\varepsilon_n)\}$, analogous arguments can be applied to show that $\PP(U\mbox{ is type 2 infected}\mid B_n)\to 1$. Since $\vep_n\searrow 0$, Proposition \ref{prop:type1} follows from this:

\begin{proof}[Proof of Proposition \ref{prop:type1}] Indeed, using that $A_n\subset \{T_{a_n}^{\sss (1)}<\mu T_{a_n}^{\sss (2)}\}$, we write
	\eqan{
	\PP\left(U \mbox{ is type 1 infected}\mid \,T_{a_n}^{\sss (1)}<\mu T_{a_n}^{\sss (2)}\right)
	&=\PP\left(U \mbox{ is type 1 infected}\mid A_n)
	\prob(A_n\mid \,T_{a_n}^{\sss (1)}<\mu T_{a_n}^{\sss (2)}\right)\\
	&\quad +\PP\left(\{U \mbox{ is type 1 infected}\}\cap A_n^c \mid \,T_{a_n}^{\sss (1)}<\mu T_{a_n}^{\sss (2)}\right),\nn
	}
Since $(T_{a_n}^{\sss (1)},T_{a_n}^{\sss (2)})\convd (V_1,V_2)$, where $(V_1,V_2)$ are independent with continuous distributions,
	\eqn{
	\lim_{n\rightarrow \infty}\prob(A_n)
	=\lim_{n\rightarrow \infty}\prob(T_{a_n}^{\sss (1)}<\mu T_{a_n}^{\sss (2)}),
	}
so that also $\prob(A_n^c \mid \,T_{a_n}^{\sss (1)}<\mu T_{a_n}^{\sss (2)})\to 0$. We conclude that $\PP\left(U \mbox{ is type 1 infected}\mid \,T_{a_n}^{\sss (1)}<\mu T_{a_n}^{\sss (2)}\right)\to 1$, as required.
\end{proof}
	
The proof of (\ref{eq:type1_given_A_n}) is divided into four parts, specified in Lemma \ref{le:1}-\ref{le:4} below. Recall that $X(u\leftrightarrow v)$ denotes the passage time between the vertices $u$ and $v$ in a one-type process with rate 1. We first observe that the one-type passage time from vertex 1 to a uniformly choosen vertex $U$ is tight:

\begin{lemma}[Tight infection times]\label{le:1}
For a uniformly chosen vertex $U$, $\PP\big(X(1\leftrightarrow U)< b_n\big)\to 1$ for all $b_n\to\infty$.
\end{lemma}

\begin{proof}
Just note that, by Proposition \ref{prop:prel}(d), the passage time between vertices 1 and $U$ converges to a proper random variable.
\end{proof}

The second lemma states roughly that, if a certain subset $\Goodn$ of the vertices is blocked, then the (one-type) passage time from vertex 2 to a randomly chosen vertex $U$ is large. To formulate this in more detail, let $\gamma,\sigma>0$ be fixed such that $\gamma<1/(3-\tau)<\sigma$. Below we will require that they are both sufficiently close to $1/(3-\tau)$. We say that a vertex $v$ of degree $D_v\geq (\log n)^\gamma$ is {\sf Good} if either $D_v\geq (\log n)^\sigma$ or if $v$ is connected to a vertex $w$ with $D_w\geq (\log n)^\sigma$ by an edge having passage time $X(e)$ at most $\varepsilon_n/2$. We let $\Goodn$ be the set of ${\sf Good}$ vertices. Furthermore, with $\CMnD$ denoting the underlying graph obtained from the configuration model and $\Omega\subset [n]$ a vertex subset, we write $\CMnD\backslash \Omega$ for the same graph but where vertices in $\Omega$ do not take part in the spread of the infection, that is, the vertices are still present in the network but are declared {\em immune} to the infection.

\begin{lemma}[Avoiding the good set is expensive]
\label{le:2}
Let the vertex $U$ be chosen uniformly at random from the vertex set. For $\gamma$ and $\sigma$ sufficiently close to $1/(3-\tau)$, there exists $b'_n\to\infty$ such that
$$
\PP\Big(\mu X(2\leftrightarrow U)\geq b'_n \mbox{ in } \CMnD\backslash\Goodn\Big)\to 1.
$$
\end{lemma}

Combining Lemma \ref{le:1} and \ref{le:2} will allow us to prove that the randomly chosen vertex $U$ is \whp type 1 infected if all vertices in $\Goodn$ are occupied by type 1. In order to show that, conditionally on $A_n$, the latter is indeed the case, we need two lemmas. The first one states roughly that there is a fast path from any vertex $u\in\Goodn$ to the exploration graph ${\rm SWG}_{a_n}^{\sss (1)}$, consisting only of vertices with degree at least $(\log n)^\gamma$. Here, for a subgraph $G$ of $\CMnD$, we define $X(u\leftrightarrow G)=\min\{X(u\leftrightarrow v):v \mbox{ is a vertex of }G\}$.

\begin{lemma}[Good vertices are found fast]
\label{le:3}
We have that
$$
	\PP\Big(\exists u\in \Goodn\text{ with }
	X(u\leftrightarrow {\rm SWG}_{a_n}^{\sss (1)})>\varepsilon_n \mbox{ in } \CMnD\backslash
	\{v\colon D_v< (\log n)^\gamma\}\Big)\to 0.
	$$
\end{lemma}
\medskip

Write ${\rm SWG}^{\sss (v)}(t)$ for the exploration graph at real time $t$ with one-type exploration starting from vertex $v$, that is, ${\rm SWG}^{\sss (v)}(t)={\rm SWG}_{k_t}^{\sss (v)}$, where $k_t=\inf\{k:T_k^{\sss (v)}\leq t\}$. The second lemma states that the one-type exploration graph emanating from vertex 2 is still small (in terms of total degree) shortly before its explosion:

\begin{lemma}[The losing type only finds low-degree vertices]
\label{le:4}
For any $k_n\to\infty$, there exist $\varepsilon_n\searrow 0$, such that, \whp
	$$
	\sum_{v\in {\rm SWG}^{\sss (2)}\big(T_{a_n}^{\sss (2)}-\varepsilon_n\big)}D_v\leq k_n.
	$$
\end{lemma}

\begin{proof}
First recall the exploration process and the corresponding approximating continuous-time branching process from Section 2. For any fixed $\vep>0$, at time $T_{a_n}^{\sss (2)}-\vep$ only an a.s.\ finite number $M=M(\vep)$ of vertices have been explored. Hence, for any $k_n\to\infty$, it is clear that we can take $\vep_n\searrow 0$ so slowly that the total degree of the explored vertices at time $T_{a_n}^{\sss (2)}-\vep_n$ is at most $k_n$.
\end{proof}
\medskip

Combining Lemma \ref{le:3} and \ref{le:4}, we can now conclude that, conditionally on $A_n$, all vertices in $\Goodn$ are \whp occupied by type 1 in the competition model:

\begin{corr}[The good vertices are all found by the winning type]
\label{cor:Goodn_1}
As $n\rightarrow \infty$,
	$$
	\PP\Big(\Goodn\mbox{ is type 1 infected }\mid A_n\Big)\to 1.
	$$
\end{corr}

\begin{proof}[Proof of Corollary \ref{cor:Goodn_1}]
First take $k_n=(\log n)^\gamma$ in Lemma \ref{le:4}, and pick $\varepsilon_n\searrow 0$ such that
	$$
	\sum_{v\in {\rm SWG}^{\sss (2)}\big(T_{a_n}^{\sss (2)}-\varepsilon_n\big)}D_v\leq (\log n)^\gamma,
	$$
where we recall that the SWG is defined based on edge weights with mean 1 and without competition. We now explore the evolution of the infection under competition, starting from vertices 1 and 2, respectively, using the coupling of the passage time variables described at the end of Section \ref{sec:prel}. We extend the notation for the exploration graph to real time in the same way as for one-type exploration, that is, ${\rm SWG}^{\sss (\underline{1},2)}(t)$ and ${\rm SWG}^{\sss (1,\underline{2})}(t)$ denote the type 1 and the type 2 part, respectively, of the exploration graph under competition at real time $t$. Note that both these graphs are increasing in $t$ and that, for a fixed $t$, we have that ${\rm SWG}^{\sss (1,\underline{2})}(\mu t)\subset {\rm SWG}^{\sss (2)}(t)$, since the type 2 infection in competition is stochastically dominated by a time-scaled one-type process (recall that the type 2 passage times under competition are multiplied by $\mu$). Combining this, we conclude that, on $A_n$,
	$$
	{\rm SWG}^{\sss (1,\underline{2})}\big(T_{a_n}^{\sss (1)}+\varepsilon_n\big)
	\subset {\rm SWG}^{\sss (2)}\big(T_{a_n}^{\sss (2)}-\varepsilon_n\big)
	$$
and hence
	\begin{equation}\label{eq:SWG2_tiny}
	\sum_{v\in{\rm SWG}^{\sss (1,\underline{2})}\big(T_{a_n}^{\sss (1)}+\varepsilon_n\big)}
	D_v\leq (\log n)^\gamma.
	\end{equation}
Since all vertices have at least degree 2, this means in particular that the number of type 2 infected vertices at time $T_{a_n}^{\sss (1)}+\varepsilon_n$ under competition is at most $(\log n)^\gamma/2$. It follows from Proposition \ref{prop:prel}(c), that ${\rm SWG}_{a_n}^{\sss (1)}$ is \whp occupied by type 1 at time $T_{a_n}^{\sss (1)}$ also in the competition model.

Now assume that there is a vertex $u\in\Goodn$ that is type 2 infected. By Lemma \ref{le:3}, \whp there exists a path connecting $u$ to ${\rm SWG}_{a_n}^{\sss (1)}$, consisting only of vertices of degree at least $(\log n)^\gamma$, such that the total passage time of the path is at most $\varepsilon_n$. Since ${\rm SWG}_{a_n}^{\sss (1)}$ is \whp occupied by type 1 at time $T_{a_n}^{\sss (1)}$ (i.e., this remains true in the presence of competition), this means that, for $u$ to be type 2 infected, one of the vertices along this path has to be type 2 infected before time $T_{a_n}^{\sss (1)}+\vep_n$. However, since all vertices on the path have degree at least $(\log n)^\gamma$, this contradicts (\ref{eq:SWG2_tiny}).
\end{proof}
\medskip

Next, we combine Lemma \ref{le:1} and \ref{le:2} with Corollary \ref{cor:Goodn_1} into a proof of (\ref{eq:type1_given_A_n}):

\begin{proof}[Proof of (\ref{eq:type1_given_A_n})]
By Lemma \ref{le:1}, \whp there exists a path $\Pi$ from vertex 1 to $U$ with $X(1\leftrightarrow U)\leq b_n$, where $b_n\to\infty$ will be further specified below. If $U$ is type 2 infected in the competition model, then the type 2 infection has to interfere with this path, that is, some vertex on the path has to be type 2 infected. This implies that $\mu X(2\leftrightarrow U)\leq b_n(1+\mu)$. Indeed, the type 2 infection has to reach the path $\Pi$ in at most time $b_n$ (otherwise the whole path will be occupied by type 1), and once it has done so, the passage time to vertex 2 is at most $\mu b_n$ (recall the coupling of the passage time variables). We obtain that
	\begin{eqnarray*}
	\PP\left(U\mbox{ is type 2 infected}\mid A_n\right)
	& \leq & \PP\left(\mu X(2\leftrightarrow U)\leq b_n(1+\mu)\mid A_n\right)\\
	& \leq & \PP\left(\mu X(2\leftrightarrow U)\leq b_n(1+\mu)
	\mbox{ in } \CMnD\backslash\Goodn \mid A_n\right)\\
 	& & +\PP\left(\exists v\in\Goodn: v \mbox{ is type 2 infected}\mid A_n\right).
	\end{eqnarray*}
With $b_n=b_n'/(1+\mu)$, where $b'_n$ is chosen to ensure the conclusion of Lemma \ref{le:2}, the first term converges to 0 by Lemma \ref{le:2}. The last term converges to 0 by Corollary \ref{cor:Goodn_1}.
\end{proof}

It remains to prove Lemmas \ref{le:2} and \ref{le:3}. We begin with Lemma \ref{le:2}:

\begin{proof}[Proof of Lemma \ref{le:2}]
We first prove a version of the lemma where $\Goodn$ is replaced by the whole set $\{v\colon D_v\geq (\log n)^\gamma\}$. According to Proposition \ref{prop:prel}(b) and (d), the passage time $X(2\leftrightarrow U)$ is \whp at most $T_{n^\rho}^{\sss (2)}+T_{n^\rho}^{\sss(U)}+\vep_n$ for some $\vep_n\searrow 0$, where $\rho$ is the exponent of the exact coupling in Proposition \ref{prop:prel}(a). If only vertices with degree smaller than $(\log n)^\gamma$ are active, then \whp
	\begin{equation}\label{eq:TU_dist}
	T_{n^\rho}^{\sss (U)}\stackrel{d}{=}\sum_{k=1}^{n^\rho}\frac{E_k}{\widetilde{S}^{{\rm\sss(trun)}}_k},
	\end{equation}
where
	$$
	\widetilde{S}^{{\rm\sss(trun)}}_k  =  D_{\sss U}\cdot
	\indic{D_U\leq (\log n)^\gamma}+\sum_{i=2}^k(\widetilde{B}_i-1)
	\cdot \indic{\widetilde{B}_i\leq (\log n)^\gamma}
	$$
for an i.i.d.\ sequence $(\widetilde{B}_i)_{i=2}^{n^\rho}$ with distribution (\ref{eq:forward_degree}), that is, a power law with exponent $\tau-1$. Let $f(n)\sim g(n)$ denote that $c\leq f(n)/g(n)\leq c'$ in the limit as $n\to\infty$ (whp when $f(n)$ is random), where $c\leq c'$ are strictly positive constants. Often, we will be able to take $c=c'$, meaning that $f(n)/g(n)$ converges to $c$ (in probability when $f(n)$ is random), but the more general definition is needed to handle the assumption (A2) on the degree distribution. We calculate that
	$$
	\E\big[(\widetilde{B}_i-1)\cdot \indic{\widetilde{B}_i\leq (\log n)^\gamma}\big]
	\sim \sum_{j=1}^{(\log n)^\gamma}j^{-(\tau-2)}\sim (\log n)^{\gamma(3-\tau)},
	$$
and that
	$$
	{\rm Var}\big((\widetilde{B}_i-1)\cdot \indic{\widetilde{B}_i\leq (\log n)^\gamma}\big)
	\leq \E\big[(\widetilde{B}_i)^2\cdot \indic{\widetilde{B}_i\leq (\log n)^\gamma}\big]
	\sim \sum_{j=1}^{(\log n)^\gamma}j^{(3-\tau)}\sim (\log n)^{\gamma(4-\tau)},
	$$
so that $\E[\widetilde{S}^{{\rm\sss(trun)}}_{k}]\sim k (\log n)^{\gamma(3-\tau)}$ and $\mathrm{Var}(\widetilde{S}^{{\rm\sss(trun)}}_{k})\sim k (\log n)^{\gamma(4-\tau)}$. Furthermore, trivially, for any $a>0$,
	$$
	T_{n^\rho}^{\sss (U)}\geq
	\sum_{k=(\log{n})^a}^{n^\rho}\frac{E_k}{k}\cdot\frac{k}{\widetilde{S}^{{\rm\sss(trun)}}_k}.
	$$
We now claim that \whp $\widetilde{S}^{{\rm\sss(trun)}}_k\leq Ck(\log n)^{\gamma(3-\tau)}$ for all $k\in[(\log{n})^a,n^\rho]$ and some constant $C$. To see this, note that $\widetilde{S}^{{\rm\sss(trun)}}_{k+1}\geq \widetilde{S}^{{\rm\sss(trun)}}_k$ so that it suffices to show that
	$$
	\PP\left(\exists l\colon \widetilde{S}^{{\rm\sss(trun)}}_{k_l}
	>Ck_l(\log n)^{\gamma(3-\tau)}\right)\to 0,
	$$
where $k_l=2^l (\log{n})^a$ and $l$ is such that $2^l (\log{n})^a\in[(\log{n})^a,n^\rho]$. We fix $l$ and $k=2^l (\log{n})^a$. With $C$ chosen such that $Ck(\log n)^{\gamma(3-\tau)}\geq 2\E[\widetilde{S}^{{\rm\sss(trun)}}_{k}]$, by the Chebyshev inequality,
	$$
	\PP\left(\widetilde{S}^{{\rm\sss(trun)}}_{k}>Ck(\log n)^{\gamma(3-\tau)}\right) 	
	\leq \PP\left(\widetilde{S}^{{\rm\sss(trun)}}_{k}
	>2\E[\widetilde{S}^{{\rm\sss(trun)}}_{k}]\right)
	\leq \frac{{\rm Var}(\widetilde{S}^{{\rm\sss(trun)}}_{k})}
	{\E[\widetilde{S}^{{\rm\sss(trun)}}_{k}]^2}
	\sim  \frac{(\log n)^{\gamma(\tau-2)}}{k}.
	$$
We substitute $k=2^l (\log{n})^a$ and use the union bound to obtain that
	$$
	\PP\left(\exists l\colon \widetilde{S}^{{\rm\sss(trun)}}_{k_l}>Ck_l(\log n)^{\gamma(3-\tau)}\right)
	\leq \sum_{l\geq 0} \frac{(\log n)^{\gamma(\tau-2)}}{k_l},
	$$
which clearly converges to 0 when $k_l=2^l (\log{n})^a>(\log{n})^a$ and $a>0$ is sufficiently large. It follows that, whp,
	$$
	T_{n^\rho}^{\sss (U)}\geq \frac{1}{C(\log n)^{\gamma(3-\tau)}}\sum_{k=(\log{n})^a}^{n^\rho}E_k/k,
	$$
where $\sum_{k=(\log{n})^a}^{n^\rho}E_k/k\sim \log n$. If $\gamma<1/(3-\tau)$, then $\kappa:=1-\gamma(3-\tau)>0$ and the desired conclusion follows with $b'_n=c(\log n)^\kappa$.

We now describe how to adapt the above arguments to obtain the statement of the lemma. Recall that a vertex $v$ of degree $D_v\geq (\log n)^\gamma$ is called {\sf Good} if either $D_v\geq (\log n)^\sigma$ or if $v$ is connected to a vertex $w$ with $D_w\geq (\log n)^\sigma$ by an edge having passage time $X(e)$ at most $\varepsilon_n/2$, and  $\Goodn$ is the set of ${\sf Good}$ vertices. When only the vertices in $\Goodn$ are inactive -- instead of the whole set $\{v:D_v\geq (\log n)^\gamma\}$ -- the denominator in (\ref{eq:TU_dist}) becomes
	$$
	\widetilde{S}^{{\rm\sss(trun)}}_k  =  D_{\sss U}\cdot
	\indic{U\not\in \Goodn}+\sum_{i=2}^k(\widetilde{B}_i-1)
	\cdot \indic{W_i\not\in \Goodn},
	$$
with $W_i$ denoting the vertex that corresponds to the forward degree $\widetilde{B}_i$. Since $k\leq n^{\rho}$, the probability that a vertex $v$ of degree $D_v\geq (\log n)^\gamma$ found in the exploration is {\sf Good} is, irrespective of all randomness up to that point, at least
	$$
	\prob({\rm Bin}(m_n, p_n)\geq 1),
	$$
with $m_n=(\log n)^\gamma-1$ and $p_n=\prob(E\leq \varepsilon_n/2) \E[J_n/L_n]$, where $J_n=\sum_{i\in[n]} D_i\indic{D_i\geq (\log n)^\sigma}-n^{\rho} (\log n)^\sigma$ and $L_n$ is the total degree of all vertices. Here, $n^{\rho} (\log n)^\sigma$ is an upper bound on the number of half-edges attached to vertices that have already been explored. Note that the knowledge that a vertex $W_i\not\in\Goodn$ gives information on the edge weights of edges connecting it to neighbors of degree at least $(\log{n})^\sigma$, but does not affect the distribution of edge weights on its other edges. Hence
	$$
	\widetilde{S}^{{\rm\sss(trun)}}_k
	\succeq \bar{S}^{{\rm\sss(trun)}}_k\equiv D_{\sss U}I_1+\sum_{i=2}^k(\widetilde{B}_i-1)
	\cdot (\indic{\widetilde{B}_i\leq (\log{n})^\gamma} +\indic{\widetilde{B}_i> (\log{n})^\gamma}I_i),
	$$
where $(I_i)_{i\geq 1}$ are i.i.d.\ Bernoulli's with success probability $\prob({\rm Bin}(m_n, p_n)=0)$ that are independent from the exponential variables $(E_i)_{i\geq 1}$ in \eqref{eq:TU_dist}. Since $\rho<1$, we can bound that
	\eqn{
	\label{pn-bd}
	p_n\geq \varepsilon_n \E[J_n/L_n] \sim \varepsilon_n (\log{n})^{-\sigma(\tau-2)}.
	}
Now we can repeat the steps in the proof of Lemma \ref{le:1}, instead using that
	\eqan{
	\E[(\widetilde{B}_i-1)I_i]&
	\sim \sum_{j=1}^{(\log n)^\gamma}j^{-(\tau-2)}+\sum_{j=(\log n)^\gamma}^{(\log{n})^\sigma}
	j^{-(\tau-2)}\prob({\rm Bin}(m_n, p_n)=0)\nn\\
	&\sim (\log n)^{\gamma(3-\tau)}
	+ (\log n)^{\sigma(3-\tau)}\prob({\rm Bin}(m_n, p_n)=0),\nn
	}
and, using \eqref{pn-bd},
$$
\prob({\rm Bin}(m_n, p_n)=0)=(1-p_n)^{m_n}\leq
\e^{-c\varepsilon_n (\log{n})^{\gamma-\sigma(\tau-2)}}.
$$
Since $\gamma<1/(3-\tau)$ and $\sigma>1/(3-\tau)$ can each be chosen as close to $1/(3-\tau)$ as we wish, we have that $\prob({\rm Bin}(m_n, p_n)=0)\leq \e^{-c\varepsilon_n (\log{n})^{\alpha}}$ for some $\alpha>0$. As a result, if $\varepsilon_n\geq c(\log\log n)^{-1}$, then $\E[(\widetilde{B}_i-1)I_i]$ obeys almost the same upper bound as $\E[(\widetilde{B}_i-1)\cdot \indic{\widetilde{B}_i\leq (\log n)^\gamma}]$ in the proof of Lemma \ref{le:2}. It is not hard to see that also ${\rm Var}((\widetilde{B}_i-1)I_i)$ obeys a similar bound as ${\rm Var}((\widetilde{B}_i-1)\cdot \indic{\widetilde{B}_i\leq (\log n)^\gamma})$. The steps in the proof of Lemma \ref{le:2} can then be followed verbatim.
\end{proof}

In order to prove Lemma \ref{le:3}, we will need the following bound, derived in \cite[(4.36)]{HofHooZna07}:

\begin{lemma}[van der Hofstad, Hooghiemstra, Znamenski (2007)]\label{le:conn_probab}
Let $\Gamma$ and $\Lambda$ be two disjoint vertex sets and write $\Gamma~\not\!\!\!\longleftrightarrow \Lambda$ for the event that no vertex in $\Gamma$ is connected to a vertex in $\Lambda$. Write $D_\Gamma$ and $D_\Lambda$ for the total degree of the vertices in $\Gamma$ and $\Lambda$, respectively, and $L_n$ for the total degree of all vertices. Furthermore, let $\prob_n$ be the conditional probability of the configuration model given the degree sequence $(D_i)_{i=1}^n$. Then,
	\begin{equation}
	\label{eq:lambda_conn_gamma}
	\prob_n(\Gamma ~\not\!\!\!\longleftrightarrow \Lambda)\leq \e^{-D_\Gamma D_\Lambda/(2L_n)}.
	\end{equation}
\end{lemma}

\begin{proof}[Proof of Lemma \ref{le:3}]
By definition of $\Goodn$, any vertex $u\in \Goodn$ is connected to a vertex $w$ with $D_w\geq (\log n)^\sigma$ by an edge with weight at most $\vep_n/2$. Write $D_{\rm max}=\max_{i\in[n]} D_i$ for the maximal degree, and denote $\mathcal{V}_{\rm max}=\{v \colon D_v=D_{\rm max}\}$.  We will show that, for each vertex $v_{\rm max}\in\mathcal{V}_{\rm max}$,
	\begin{equation}
	\label{eq:w_max}
	\PP\left(D_w\geq (\log n)^\sigma, X(w\leftrightarrow v_{\rm max})>\varepsilon_n/4
	\mbox{ in } \CMnD\backslash \{v\colon D_v< (\log n)^\gamma\}\right)=o(1/n),
	\end{equation}
and
	\begin{equation}\label{eq:1_max}
	\PP(X(1\leftrightarrow v_{\rm max})>T_{a_n}^{\sss (1)}+\varepsilon_n/4 \mbox{ in }
	\CMnD\backslash \{v\colon D_v< (\log n)^\gamma\})=o(1).
	\end{equation}
Lemma \ref{le:3} follows from this by noting that
	\eqan{
	& \PP(\exists w\colon D_w\geq (\log n)^\sigma,
	X(w\leftrightarrow {\rm SWG}_{a_n}^{\sss (1)})>\vep_n/2 \mbox{ in } \CMnD\backslash \{v\colon D_v< (\log n)^\gamma\})\nn\\
	& \qquad \leq n\PP(D_w\geq (\log n)^\sigma,X(w\leftrightarrow v_{\rm max})>\vep_n/4 \mbox{ in } \CMnD\backslash \{v:D_v< (\log n)^\gamma\})\nn\\
	&\qquad \quad + \PP(X(1\leftrightarrow v_{\rm max})>T_{a_n}^{\sss (1)}+\vep_n/4 \mbox{ in } \CMnD\backslash \{v:D_v< (\log n)^\gamma\})=o(1).\nn
	}
To prove (\ref{eq:w_max}), we will construct a path $v_0,\ldots,v_m$ with $v_0=w$ and $v_m=v_{\rm max}$ and with the property that the passage time for the edge $(v_i,v_{i+1})$ is at most $(\log D_{v_i})^{-1}$, while $D_{v_i}\geq (\log n)^{\alpha_i}$ where $\alpha_i$ grows exponentially in $i$. The total passage time along the path is hence smaller than
	\begin{equation}\label{eq:sum}
	\sum_{i=1}^m \frac{1}{\log D_{v_i}}\leq\sum_{i=1}^m \frac{1}{\log\big((\log n)^{\alpha_i}\big)}
	\leq\frac{1}{\log\log n}\sum_{i=1}^m\frac{1}{\alpha_i}
	=O\left(\frac{1}{\log\log n}\right),
	\end{equation}
which is smaller than $\varepsilon_n/4$ since $\vep_n\geq c(\log\log{n})^{-1}$ where $c>0$ can be chosen appropriately.

Say that an edge emanating from a vertex $u$ is \emph{fast} if its passage time is at most $1/(\log D_u)$ and write $D_u^{\rm fast}$ for the number of such edges. Note that
	$$
	\E[D_u^{\rm fast}\mid D_u]=D_u[1-\e^{-1/\log D_u}]
	=\frac{D_u}{\log D_u}\left[1+O\left(\frac{1}{\log D_u}\right)\right]
	$$
and that, by standard concentration inequalities,
	$$
	\prob(D_u^{\rm fast}\leq D_u/[2\log D_u]\mid D_u)\leq \e^{-c D_u/\log D_u}.
	$$	
Indeed, conditionally on $D_u=d$, we have that $D_u^{\rm fast}\stackrel{d}{=} {\rm Bin}(d,1-\e^{-1/\log d})$ and, for any $p$, it follows from standard large deviation techniques that
	\eqn{
	\label{bin-conc}
	\prob({\rm Bin}(d,p)\leq pd/2)\leq \e^{-pd(1-\log 2)/2},
	}
see e.g.\ \cite[Corollary 2.18]{Remco_notes}. In particular, if $D_u\geq (\log n)^\sigma$ with $\sigma>1$, we obtain that
	\eqn{
	\label{eq:many_fast}
	\prob\Big(\exists u\colon D_u\geq (\log{n})^{\sigma}, D_u^{\rm fast}\leq D_u/[2\log D_u]\Big)	 
	\leq n\e^{-c(\log{n})^{\sigma}/\log((\log{n})^{\sigma})}
	=o(1).
	}
Thus, we may assume that $D_u^{\rm fast}>D_u/[2\log(D_u)]$ for any $u$ with $D_u\geq (\log{n})^{\sigma}$.

Write $\Lambda_i=\{u\colon D_u\geq \eta_i\}$, where $\eta_i$ will be defined below and shown to equal $(\log n)^{\alpha_i}$ for an exponentially growing sequence $(\alpha_i)_{i\geq 1}$. Furthermore, let $\Gamma(u)$ denote the set of fast half-edges from a vertex $u$. We now construct the aforementioned path connecting $w$ and $\mathcal{V}_n^{\rm max}$ iteratively, by setting $v_0:=w$ and then, given $v_i$, defining $v_{i+1}\in\Lambda_{i+1}$ to be the vertex with smallest index such that a half-edge in $\Gamma(v_i)$ is paired to a half-edge incident to $v_{i+1}$. We need to show that, with sufficiently high probability, such vertices exist all the way up until we have reached $\mathcal{V}_n^{\rm max}$. This will follow basically by observing that, for any vertex $u_i\in\Lambda_i$, we have by Lemma \ref{le:conn_probab} that
	\begin{equation}\label{li_conn_lai1}
	\prob_n(\Gamma(u_i)\not\leftrightarrow \Lambda_{i+1})
	\leq \E_{n}\left[\e^{-D_{u_i}^{\rm fast}D_{\Lambda_{i+1}}/(2L_n)}\right],
	\end{equation}
where the expectation is over the randomness in the edge weights used for defining $D_{u_i}^{\rm fast}$, and then combining this with suitable estimates of the exponent.

First we define the sequence $(\eta_i)_{i\geq 1}$. To this end, let $\eta_1=(\log n)^{\sigma}$ and define $\eta_{i}$ for $i\geq 2$ recursively as
	\eqn{
	\label{eta-recursion}
	\eta_{i+1}=\Big(\frac{\eta_i}{\log{n}}\Big)^{(1-\delta)/(\tau-2)},
	}
where $\delta\in(0,1)$ will be determined below. To identify $(\eta_i)_{i\geq 1}$, write $\eta_i=(\log{n})^{\alpha_i}$ and check that $(\alpha_i)_{i\geq 1}$ satisfy $\alpha_1=\sigma$ and the recursion 
	$$
	\alpha_{i+1}=\frac{1-\delta}{\tau-2} \alpha_i-\frac{1-\delta}{\tau-2}.
	$$
As a result, when $\delta<3-\tau$ so that $(1-\delta)>(\tau-2)$, we can bound
	\eqan{
	\alpha_i&=\alpha_1 \Big(\frac{1-\delta}{\tau-2} \Big)^{i-1}
	-\sum_{j=1}^{i-1} \Big(\frac{1-\delta}{\tau-2} \Big)^{j}\nn\\
	&=\alpha_1 \Big(\frac{1-\delta}{\tau-2} \Big)^{i-1}-\frac{\Big(\frac{1-\delta}{\tau-2} \Big)^{i-1}-1}{1-\frac{\tau-2}{1-\delta}}\nn\\
	&=\Big[\alpha_1-\frac{1}{1-\frac{\tau-2}{1-\delta}}\Big]\left(\frac{1-\delta}{\tau-2}\right)^i
	+\frac{1}{1-\frac{\tau-2}{1-\delta}},\nn
	}
which is strictly increasing and grows exponentially as long as $\alpha_1=\sigma>(1-\delta)/[3-\tau-\delta]$, that is, $\delta<[\sigma(3-\tau)-1]/(\sigma-1)$. Since $\sigma>1/(3-\tau)>1$, this is indeed possible. With $\sigma>1/(3-\tau)$, we then see that $i\mapsto \alpha_{i}$ is strictly increasing and grows exponentially for large $i$.

We next proceed to estimate the exponent in (\ref{li_conn_lai1}). We first recall some facts proved in \cite{HofHooZna07}. First, under the assumption of our paper, it is shown in \cite[(A.1.23)]{HofHooZna07} that there exist $a>1/2$ and $\chi>0$ such that
$$
\prob(|L_n-n\expec[D]|>n^a)\leq n^{-\chi}.
$$
Further, in \cite[Lemma A.1.3]{HofHooZna07}, it is shown that for every $b<1/(\tau-1)$, there exists a $\xi>0$ such that
\begin{equation}\label{eq:G}
\prob\big(\exists x\leq n^{b}\colon  |G_n(x)-G(x)|\geq n^{-\xi}[1-G(x)]\big)\leq n^{-\xi},
\end{equation}
where
$$
G_n(x)=\frac{1}{L_n}\sum_{i\in[n]} D_i\indic{D_i\leq x},
\qquad
\text{and}
\qquad
G(x)=\frac{\expec[D\indic{D\leq x}]}{\expec[D]}.
$$
We will work with $\prob_n$, and condition the degrees to be such that the event $F_n$ occurs, where
\begin{eqnarray*}
F_n&=\big\{|L_n-n\expec[D]|\leq n^a\big\}\cap
\big\{\forall x\leq n^b\colon  |G_n(x)-G(x)|\leq n^{-\xi}[1-G(x)]\big\}\\
&\qquad
\cap\big\{D_u^{\rm fast}\geq D_u/[2\log(D_u)]~\forall u
\text{ with\ }D_u\geq (\log{n})^{\sigma}\big\} ,\nn
\end{eqnarray*}
so that in particular $\prob(F_n^c)\leq n^{-\xi}+n^{-\chi}+o(1)=o(1)$.

On the event $F_n$, as long as $\eta_{i+1}\leq n^{(1-\delta/2)/(\tau-1)}$ (this is to ensure that (\ref{eq:G}) is valid with $b=\eta_{i+1}$)
	$$
	\frac{D_{\Lambda_{i+1}}}{L_n}=\frac{1}{L_n}\sum_{v\in[n]} D_v\indic{D_v>\eta_{i+1}}
	\geq c\E[D\indic{D>\eta_{i+1}}]\geq c\eta_{i+1}^{-(\tau-2)}.
	$$
Furthermore, for every vertex $u_i\in \Lambda_i$, we obtain as in (\ref{eq:many_fast}) that
	$$
	D_{u_i}^{\rm fast}\geq \frac{D_{u_i}}{2\log D_{u_i}}\geq \frac{\eta_i}{2\log \eta_i},
	$$
where the first inequality holds with probability $1-o(1/n)$. Combining these two estimates and applying Lemma \ref{le:conn_probab} gives that
	$$
	\prob_n(\Gamma(u_i)~\not\!\!\!\longleftrightarrow \Lambda_{i+1})
	\leq \exp\{-c(\eta_i/\log(\eta_i)) \eta_{i+1}^{-(\tau-2)}\}.
	$$
Using \eqref{eta-recursion} and the fact that $\eta_i=(\log n)^{\alpha_i}$ it follows that
\begin{eqnarray*}
	\prob_n(\Gamma(u_i)~\not\!\!\!\longleftrightarrow \Lambda_{i+1})
	& \leq & \exp\{-c(\eta_i^{\delta}/\log(\eta_i)) \cdot (\log{n})^{(1-\delta)}\}\\
	& \leq & \exp\{-c((\log n)^{1+\delta(\alpha_i-1)}/\log(\eta_i)) \},
\end{eqnarray*}
which is $o(n^{-a})$ for any $a>0$. Taking $a>3$, this implies that, as long as $\eta_i\leq n^{(1-\delta/2)/(\tau-1)}$,
	$$
	\prob_n(\exists i \text{ and }u_i\in \Lambda_i\colon \Gamma(u_i)~\not\!\!\!\longleftrightarrow  \Lambda_{i+1})
	=o(1/n).
	$$
Hence, as long as $\eta_i\leq n^{(1-\delta/2)/(\tau-1)}$, the probability that the construction of the path $(v_i)_{i\geq 1}$ fails in some step is $o(1/n)$.

Let $i^*=\max\{i\colon \eta_i\leq n^{(1-\delta/2)/(\tau-1)}\}$ be the largest $i$ for which $\eta_i$ is small enough to guarantee that the failure probability is suitably small. The path $v_0, \ldots, v_{i^*}$ then has the property that $D_{v_i}\geq (\log n)^{\alpha_i}$ and the passage time on the edge $(v_i,v_{i+1})$ is at most $(\log D_{v_i})^{-1}$, as required. To complete the proof of \eqref{eq:w_max}, it remains to show that, with probability $1-o(1/n)$, the vertex $v_{i^*}$ has an edge with vanishing weight connecting to the vertex $v_{\rm max}\in\mathcal{V}_{\rm max}$.

To this end, note that, using (\ref{eta-recursion}) and the definition of $i^*$, we can bound
	$$
	D_{v_{i^*}}\geq \eta_{i^*}\geq \eta_{i^*+1}^{(\tau-2)/(1-\delta)}\log n
	\geq n^{\frac{(1-\delta/2)(\tau-2)}{(1-\delta)(\tau-1)}}.
	$$
Furthermore, $D_{\rm max}\geq n^{(1-h\delta)/(\tau-1)}$ with probability $1-o(1/n)$ for any $h>0$, since
	\begin{equation}\label{eq:max_calc}
	\PP(D_{\rm max}\geq x) \leq 1-(1-cx^{-(\tau-1)})^n,
	\end{equation}
which decays stretched exponentially for $x=n^{(1-h\delta)/(\tau-1)}$. Define $\psi=[(1-\delta/2)(\tau-2)]/[(1-\delta)(\tau-1)]$ and $\phi=(1-h\delta)/(\tau-1)$, where $h$ will be specified below. Assuming that $D_{v_{i^*}}=n^{\psi}$ and $D_{\rm max}=n^{\phi}$, the number $H$ of (multiple) edges between $v_{i^*}$ and $v_{\rm max}$ is hypergeometrically distributed with
$$
\E[H]=n^{\psi}\cdot\frac{n^\phi}{n-n^{\psi}}\sim n^{\psi+\phi-1},
$$
where
	$$
	\psi+\phi-1=\frac{\delta}{2(1-\delta)(\tau-1)}[\tau+2h\delta-2(1+h)],
	$$
which is positive as soon as $h<(\tau-2)/(1-\delta)$. It is not hard to see -- e.g.\ by coupling $H$ to a binomial variable and using (\ref{bin-conc}) -- that $\PP\left(H\leq \E[H]/2\right)\leq \e^{-cn^{\psi+\phi-1}}$. Hence, with probability $1-o(1/n)$, the vertex $v_{i^*}$ is connected to $v_{\rm max}$ by at least $\E[H]/2\sim n^{\psi+\phi-1}$ edges. Let $(E_i)_{i\geq 1}$ be an i.i.d.\ sequence of Exp(1)-variables. The probability that all edges connecting $v_{i^*}$ and $v_{\rm max}$ have passage time larger than $1/\log n$ is then bounded from above by
	$$
	\PP\big(E_i>1/\log n\big)^{n^{\psi+\phi-1}}=\e^{-n^{\psi+\phi-1}/\log n}
	=o(1/n).
	$$
This completes the proof of \eqref{eq:w_max}.

To prove (\ref{eq:1_max}), first note that it follows from \cite[Lemma A.1]{RGSc}, that the number of infected vertices at time $T_{a_n}^{\sss (1)}$ is \whp larger than $m_n$ for any $m_n$ with $m_n/a_n\to 0$, and that, by Proposition \ref{prop:prel}(a), there exist $\rho>0$ such that the degrees $(B_i)_{i=2}^{n^\rho}$ of the $n^\rho$ first vertices that were infected are \whp equal to an i.i.d.\ collection $(\widetilde{B}_i)_{i=2}^{n^\rho}$ with distribution (\ref{eq:forward_degree}). A calculation analogous to (\ref{eq:max_calc}) yields that $\max\{B_2,\ldots,B_{n^\rho}\}\geq n^{\rho(1-\delta)/(\tau-2)}$ \whp for any $\delta\in(0,1)$. The vertex with maximal degree at time $T_{a_n}^{\sss(1)}$ can now be connected to $v_{\rm max}$ by a path constructed in the same way as in the proof of (\ref{eq:w_max}). Note that in this case we have $\eta_1=n^{\rho(1-\delta)/(\tau-2)}$, which gives $\eta_i=n^{\rho\zeta^i}/(\log n)^{\zeta^{i-1}}$ with $\zeta=(1-\delta)/(\tau-2)$. This means that the bound on the passage time for the path is of order $1/\log n$, which is even smaller than the required $1/\log\log n$.
\end{proof}

\section{Proof of Theorem \ref{th:main}(b)}
\label{sec:main_b}
In this section, we prove Theorem \ref{th:main}(b). Throughout this section, we deal with the competition process, and explore the competition from the two vertices 1 and 2 simultaneously. Let $\Totm$ denote the time when the SWG from these two vertices consists of $m$ vertices (recall the definition (\ref{eq:Rm}) of $R_m$). Furthermore, write $\mathcal{W}_n$ for the type that occupies the largest number of vertices at time $\Totan$ and $\mathcal{L}_n$ for the type that occupies the smallest number of vertices. We will show that $\mathcal{W}_n$ wins with probability 1 as $n\to\infty$ and that $\mathcal{L}_n$ is hence asymptotically the losing type. Our first result is that $\Totan$ converges to the minimum of the explosion times $V_1$ and $\mu V_2$ of the one-type exploration processes, and that the asymptotic number $\Nlos^*$ of vertices that are then occupied by type $\mathcal{L}_n$ is finite. In the rest of the section, we then prove that the asymptotic number $\Nlos^{**}$ of vertices occupied by type $\mathcal{L}_n$ after time $\Totan$ is also almost surely finite.

We start by introducing some notation. Let $\mathcal{W}$ and $\mathcal{L}$ denote the winning and the losing type, respectively, in the limit as $n\to\infty$. Also let $\muwin=\mu$ when the winning type is type 2, and $\muwin=1$ otherwise, and similarly $\mulos=\mu$ when the losing type is type 2, and $\mulos=1$ otherwise. According to Theorem \ref{th:main}(a), asymptotically type 1 wins with probability $\PP(V_1<\mu V_2)$ and type 2 with probability $\PP(V_1>\mu V_2)$. Hence $\muwin$ is equal to 1 with probability $\PP(V_1<\mu V_2)$ and equal to $\mu$ with probability $\PP(V_1>\mu V_2)$. Finally, let $(E_j^{\sss(1)})_{j\geq 1},(E_j^{\sss(2)})_{j\geq 1}$ denote two sequences of i.i.d.\ exponential random variables with mean 1, and $(\widetilde{S}_j^{\sss(1)})_{j\geq 1},(\widetilde{S}_j^{\sss(2)})_{j\geq 1}$ two i.i.d.\ sequences of the random walk describing the asymptotic number of unexplored half-edges attached to the SWG in a one-type exploration process, cf.\  Section 2. Then,
	$$
	V_i=\sum_{j=1}^{\infty} E_j^{\sss(i)}/S_j^{\sss(i)}
	$$
denote the explosion times of the corresponding continuous time branching process (CTBP). Let
	$$
	\Vwin=\muwin\sum_{j=1}^{\infty} \Ewin_j/\Swin_j,
	\qquad
	\Vlos=\mulos\sum_{j=1}^{\infty} \Elos_j/\Slos_j.
	$$	
Then, $V_1\wedge (\mu V_2)=\Vwin$ is close to the time when the winning type finds vertices of very high degree. The random variable $\Vlos$ does not have such a simple interpretation in terms of the competition process, since the winning type starts interfering with the exploration of the losing type before time $\Vlos$. The main aim of this section is to describe the exploration of the winning and losing types after time $\Vwin$, where the CTBP approximation breaks down and the species start interfering. The relation between the number of vertices found by the losing kind and $\Vwin$ is described in the following lemma:
	
\begin{lemma}[Status at completion of the CTBP phase]
\label{le:tt_explosion}
Let $\NLna=\max\{m\colon T^{\sss (\sss\mathcal{L}_n)}_{\sss R_m} \leq \Totan\}$. Then, as $n\rightarrow \infty$,
	$$
	(\Totan, \NLna)\convd (\Vwin, \Nlos^*),
	$$
where
	\begin{equation}
	\label{Nlos-def}
	\Nlos^*\stackrel{d}{=}\max\big\{m\colon \mulos \sum_{j=1}^{m} \Elos_j/\Slos_j \leq \Vwin\big\}.
	\end{equation}
\end{lemma}

\proof
By definition, the number of vertices occupied by type $\mathcal{W}_n$ at time $\Totan$ is in the range $(a_n/2,a_n]$. Furthermore, by Proposition \ref{prop:prel}(c), the set of type 1 and type 2 infected vertices, respectively, are whp disjoint at this time, that is, none of the infection types has then tried to occupy a vertex that was already taken by the other type. Up to that time, the exploration processes started from vertex 1 and 2, respectively, hence behave like in the corresponding one-type processes. The asymptotic distributions of $\Totan$ and $\NLna$ follow from the characterization (\ref{eq:Tm_distr}) of the time $T_m$ in a one-type process and the convergence result in Proposition \ref{prop:prel}(c).
\qed\medskip

The next result describes how vertices are being found by type $\mathcal{W}_n$ after time $\Totan$.
We will see that at time $\Totan+t$, a positive proportion of the vertices will be found by the winning type. To describe how the winning type sweeps through the graph, we need some notation. Write $\Ntk$ for the fraction of vertices that have degree $k$ and that have been captured by type $\mathcal{W}_n$ at time $\Totan+t$, that is,
	$$
	\Ntk=\#\{v\colon D_v=k \mbox{ and $v$ is infected by type $\mathcal{W}_n$ at time }\Totan+t\}/n.
	$$
Further, for an edge $e=xy$ consisting of two half-edges $x$ and $y$ that are incident to vertices $U_x$ and $U_y$, we say that $e$ {\em spreads the winning infection at time $s$} when $U_x$ (or $U_y$) is type $\mathcal{W}_n$ infected at time $s$, and $U_y$ (or $U_x$) is then $\mathcal{W}_n$ infected at time $s$ through the edge $e$. Then we let $\Ltwin$ denote the number of edges that have spread the type $\mathcal{W}_n$ infection by time $s$, i.e.,
	$$
	\Ltwin=\#\{e\colon \mbox{ $e$ has spread the winning infection at time }\Totan+t\},
	$$
and $\barLtwin=\Ltwin/[L_n/2]$ is the proportion of edges that have spread the winning infection.

The essence of our results is that $\Ntk$ and $\barLtwin$ develop in the same way as in a one-type process with type $\mathcal{W}_n$ {\em without competition}. Indeed, $\Totan$ can be interpreted as the time when the super-vertices have been found by type $\mathcal{W}_n$ and, after this time, type $\mathcal{W}_n$ will start finding vertices very quickly, which will make it hard for type $\mathcal{L}_n$ to spread. Recall that $\muwin$ denotes the mean passage time per edge for the winning type in the limit as $n\to\infty$. Also define
	$$
	V(k)=\sum_{j=0}^{\infty} E_j/S_j(k),
	\qquad
	\text{where}
	\qquad
	S_j(k)=k+\sum_{i=1}^j (\widetilde{B}_i-1),
	$$
and $(\widetilde{B}_i)_{i\geq 1}$ is an i.i.d.\ sequence with law \eqref{eq:forward_degree}. Recall that $D^\star$ denotes a size-biased version of a degree variable.
\begin{prop}[Fraction of fixed degree winning type vertices and edges at fixed time]
\label{prop-degree-time}
As $n\rightarrow \infty$,
	\eqn{
	\label{Ntk-conv}
	\Ntk\convp \PP(\muwin V(k)\leq t)\PP(D=k),
	}
and
	\eqn{
	\label{Ltwin-conv}
	\barLtwin \convp \PP\Big(\muwin\big(E+\widetilde V_a\wedge\widetilde V_b\big)\leq t\Big),
	}
where $(\widetilde V_a,\widetilde V_b)$ are two independent copies of $V(D^\star-1)$ and $E$ is an exponential random variable with mean 1.
\end{prop}
\medskip

The proof of Proposition \ref{prop-degree-time} is deferred to the end of this section. We first complete the proof of Theorem \ref{th:main}(b) subject to it. To this end, we grow the SWG of type $\mathcal{L}_n$ from size $\NLna$ onwards. At this moment, \whp the type $\mathcal{L}_n$ has not yet tried to occupy a vertex that was already taken by type $\mathcal{W}_n$. However, when we grow the SWG further, then type $\mathcal{W}_n$ will grow very quickly due to its explosion. We will show that the growth of type $\mathcal{L}_n$ is thus delayed to the extent that it will only conquer finitely many vertices. An important tool in proving this rigorously is a stochastic process $(S_m')_{m\geq 0}$ keeping track of the number of unexplored half-edges incident to the SWG of the losing type.

Recall that, by the construction of the two-type exploration process described in Section \ref{sec:prel}, the quantity $S_{\sss R_j}^{\sss(\mathcal{W}_n, \underline{\mathcal{L}_n})}$ represents the number of half-edges incident to the SWG of type $\mathcal{L}_n$ when the SWG contains precisely $j$ vertices. Write $R_{\NLna}=R^\ast_n$ and define $S_0'(n)=S_{R^\ast_n}^{\sss(\mathcal{W}_n,\underline{\mathcal{L}_n})}$ and $T_{0}'(n)=0$. We then grow the SWG of type $\mathcal{L}_n$ one edge at a time by pairing the half-edge with minimal remaining edge weight to a uniform half-edge that has not yet been paired. Denote the half-edge of minimal weight in the $m$th step by $x_m$ and the half-edge to which it is paired by $P_{x_m}$, and recall that $U_y$ denotes the vertex incident to the half-edge $y$. Of course, it is possible that  $U_{P_{x_m}}$ is already infected, and then the SWG of the losing type does not grow.

The sequences $(T_m'(n))_{m\geq 0}$ and $(S_m'(n))_{m\geq 0}$ are constructed recursively in that $T_m'(n)-T_{m-1}'(n)=\mulos E_m'/S_{m-1}'(n)$ for an i.i.d.\ sequence $(E_m')_{m\geq 0}$ of exponential variables with parameter 1 that is independent of all previous randomness, and
	\begin{equation}
	\label{eq:Sm_rec}
	S_m'(n)-S_{m-1}'(n)=B_m'(n)-1,
	\end{equation}
where $B_m'(n)$ denotes $D_{U_{P_{x_m}}}-1$ when $U_{P_{x_m}}$ is not already infected, while $B_m'(n)=0$ otherwise. Our aim is to identify the scaling limit of $(T_m'(n),S_m'(n))_{m\geq 0}$.

To this end, we define $S_0'=S_{\sss\Nlos^*}^{\sss(\mathcal{W},\underline{\mathcal{L}})}$ and $T_0'=0$, where $\Nlos^*$ is given by \eqref{Nlos-def}. Further, for $m\geq 1$, again define $(T_m',S_m')_{m\geq 0}$ recursively by $T_m'-T_{m-1}'=E_m'/S_{m-1}'$ for an i.i.d.\ sequence $(E_m')_{m\geq 0}$ of exponential variables with parameter 1 independent of all previous randomness, and $(S_m')_{m\geq 0}$ is defined recursively by $S_m'-S_{m-1}'=B_m'-1$, where, conditionally on $T_{m-1}'$ and for all $k\geq 1$
	\eqn{
	\label{Bm'-def}
	\prob(B_m'=k\mid T_m'=t)
	=\prob(D^\star=k+1)\frac{\prob(\muwin V(k+1)>t)}
	{\prob\Big(\muwin\big(E+\widetilde V_a\wedge\widetilde V_b\big)> t\Big)},
	}
while $\prob(B_m'=0\mid T_m'=t)=1-\sum_{k\geq 1} \prob(B_m'=k\mid T_m'=t)$.

\begin{remark}[Edge-weight distribution vs.\ weights on half-edges]
\label{rem-edge-weight} In the above construction, we explore from vertex 1 and 2 simultaneously, and search for the minimal weight among unexplored half-edges of the loosing type. This half-edge is then paired to a randomly chosen second half-edge, and the passage time of the resulting edge should then be given by the minimal weight, that is, the second half-edge should not be assigned any weight at all. The careful reader may note that this is not the case in the above construction when the second half-edge belongs to a vertex that is already infected. In that case, however, the edge will never be used to transmit infection (indeed, such edges are not included in the SWG defined in Section 2) and its assigned passage time is hence unimportant for the competition process. We remark that, when first exploring one type up to a fixed time or size, the above problem does not apply, see \cite{RGSc} for details.
\end{remark}
\medskip
	
The following lemma shows that $(T_m',S_m')_{m\geq 0}$ is indeed the limit in distribution of the process $(T_m'(n),S_m'(n))_{m\geq 0}$. In its statement, we use $\field_{\Totan}$ for the $\sigma$-field of the exploration of the two competing species up to time $\Totan$:

\begin{lemma}[Exploration of the losing type beyond explosion of the winning type]
\label{lem-expl-losing}
Conditionally on $\field_{\Totan}$, and for all $m\geq 1$, as $n\rightarrow \infty$,
	$$
	(T_l'(n),S_l'(n))_{l=0}^m\convd (T_l',S_l')_{l=0}^m.
	$$
\end{lemma}

\proof We prove the claim by induction on $m$. The statement for $m=0$ follows from Lemma \ref{le:tt_explosion}, since $T_0'(n)=T_0'=0$, and $S_0'(n)=S_{R^\ast_n}^{\sss(\mathcal{W}_n,\underline{\mathcal{L}_n})}$ with $R^\ast_n=\NLna$. The latter converges in distribution by Lemma \ref{le:tt_explosion}.

To advance the induction claim, we introduce some further notation.
Let $\field_t'$ be the $\sigma$-field generated by $\field_{\Totan}$ together with $(T_l'(n),S_l'(n))_{l=0}^m$ for all $l$ such that $T_l'(n)\leq t$. Then, conditionally on $\field_t'$,
	$$
	T_{m}'(n)-T_{m-1}'(n)\sim \mulos E_m'/S_{m-1}'(n),
	$$
where $E_m'$ is an exponential random variable independent of all other randomness. Since, by the induction hypothesis, $S_{m-1}'(n)\convd S_{m-1}'$, conditionally on $\field'_{T_{m-1}'(n)}$ we also have that $T_{m}'(n)-T_{m-1}'(n)\convd \mulos E_m'/S_{m-1}'$. This advances the claim for $T_{m}'(n)$. For $S_{m}'(n)$, we note that
$B_m'(n)=k$ precisely when the half-edge that is found is paired to a half-edge of a vertex of degree $k+1$ that is not yet infected. The number of vertices that is type $\mathcal{L}_n$ infected is bounded by $R^\ast_n+m-1$, so this is negligible. Therefore, writing $N_n^{\sss(k+1)}$ for the total number of vertices of degree $k+1$,
	$$
	\prob(B_m'(n)=k\mid \field'_{T_{m-1}'(n)}, T_m'(n)=t)
	=\frac{(k+1)[N_n^{\sss(k+1)}-\Ntkwin]}{L_n-\Ltwin}(1+\op(1)).
	$$
We rewrite this as
	$$
	\prob(B_m'(n)=k\mid \field'_{T_{m-1}'(n)}, T_m'(n)=t)=\frac{(k+1)}{(L_n/n)}\frac{\bar{N}_n^{\sss(k+1)}-\barNtkwin}{1-\barLtwin}(1+\op(1)),
	$$
where $\bar{N}_n^{\sss(k+1)}=N_n^{\sss(k+1)}/n$ denotes the proportion of vertices with degree $k+1$.
By Proposition \ref{prop-degree-time}, this is equal to
	\eqan{
	&\prob(B_m'(n)=k\mid \field'_{T_{m-1}'(n)}, T_m'(n)=t)\nn\\
	&\qquad=\frac{(k+1)}{\expec[D]}\frac{\prob(D=k+1)-\PP(\muwin V(k+1)\leq t)\PP(D=k+1)}{1-\PP\Big(\muwin\big(E+\widetilde V_a\wedge\widetilde V_b\big)\leq t\Big)}(1+\op(1))\nn\\
	&\qquad\convp \prob(D^\star=k+1)\frac{\prob(\muwin V(k+1)>t)}
	{\prob\Big(\muwin\big(E+\widetilde V_a\wedge\widetilde V_b\big)> t\Big)}=\prob(B_m'=k\mid T_m'=t),\nn
	}
as required. This shows that, conditionally on $\field'_{T_{m-1}'(n)}$, the law of $(T_{m}'(n)-T_{m-1}'(n), B_m'(n))$ converges to \eqref{Bm'-def}. Hence, this advances the induction and thus proves the claim.\qed\medskip

Denote $H'(n)=\max\{m\colon S_m'(n)\geq 1\}$ and $H'=\max\{m\colon S_m'\geq 1\}$. In the following corollary, we show that $H'(n)\convd H'$.

\begin{corr}[Convergence of hitting of zero]
\label{lem-hitting-time-zero}
For all $m\geq 1$, as $n\rightarrow \infty$,
	$$
	\prob(H'(n)\leq m\mid \field_{\Totan})\rightarrow \prob(H'\leq m\mid \field_{\Vwin}).
	$$
Therefore, $H'(n)\convd H'$, where $H'$ is possibly defected.
\end{corr}

\proof It suffices to realize that the event $\{H'(n)\leq m\}$ is measurable with respect to $(T_l'(n),S_l'(n))_{l=0}^m$. Then the claim follows from Lemma \ref{lem-expl-losing}.
\qed
\bigskip

Note that $B_m'$ in (\ref{Bm'-def}) has infinite mean when we condition on $T_m'=t=0$, which implies that initially many of its values are large. This is the problem that we need to overcome in showing that the number of vertices found by the losing type is finite. As it turns out, conditionally on $T_m'=t$, the mean of $B_m'$ decreases as $t$ increases, and becomes smaller than 1 for large $t$, so this saves our day. In order to prove this, we first need some results on the process $V(k)$; see part (a) of the below lemma. In part (b), we also include an asymptotic characterization that will imply \eqref{asympt-prob-win} in Theorem \ref{th:main-multiple}(c).

\begin{lemma}[Bounds and asymptotics for $V(k)$]\label{lem-V(k)-asymp}\leavevmode
\begin{itemize}
\item[\rm{(a)}] The law of $V(k)$ is related to that of $V(1)$ by
	\eqn{
	\label{V(k)-rec}
	\prob(V(k)>t)=\prob(V(1)>t)^k,\qquad k\geq 1, t\geq 0.
	}
Further, for all $t\geq 0$,
\eqn{
\label{EVaVb_bd}
\prob\left(E+\widetilde V_a\wedge\widetilde V_b> t\right)\geq \prob(V(1)>t)^2.
}
\item[\rm{(b)}] Assume that (A2') holds. As $k\rightarrow \infty$,
	\eqn{
	\label{Y-def}
	k^{3-\tau}V(k) \convd Y\equiv \int_0^{\infty} 1/(1+Q_t)dt,
	}
where $(Q_t)_{t\geq 0}$ is a $(\tau-2)$-stable motion. Further, $\expec[Y]<\infty$.
\end{itemize}
\end{lemma}
	
\proof Starting with (a), the relation \eqref{V(k)-rec} follows since $V(k)$ is the explosion time starting from $k$ individuals, which is the minimum of the explosion times of $k$ i.i.d.\ explosion times starting from 1 individual, i.e.,
	$$
	V(k)\stackrel{d}{=}\min_{i=1}^k V_i,
	$$
where $(V_i)_{i\geq 1}$ are i.i.d.\ with law $V(1)$ and $E$ is exponential with parameter 1. From this, \eqref{V(k)-rec} follows immediately. For \eqref{EVaVb_bd}, write $G(t)=\prob(V(1)>t)$ and note that $V(1)\stackrel{d}{=}E+\min_{i=1}^{D^{\star}-1} V_i$, where again $(V_i)_{i\geq 1}$ are i.i.d.\ with law $V(1)$. Thus, conditioning on $E$ and $D^{\star}$, and using \eqref{V(k)-rec}, leads to
	\eqn{
	\label{V(1)-rec}
	G(t)=\e^{-t}+\int_0^t \e^{-s} \expec[G(t-s)^{D^{\star}-1}]ds.
	}
Furthermore, since $\widetilde{V}_a$ and $\widetilde{V}_b$ are i.i.d.\ with the same distribution as $V(D^{\star}-1)\stackrel{d}{=}\min_{i=1}^{D^{\star}-1} V_i$, we similarly obtain that
\eqn{
\label{edge_rec}
\prob\left(E+\widetilde V_a\wedge\widetilde V_b> t\right)=\e^{-t} + \int_0^t \e^{-s}\expec[G(t-s)^{D^{\star}-1}]^2ds.
}
By the Cauchy-Schwarz inequality (where we split $\e^{-s}=\e^{-s/2}\e^{-s/2}$ on the left hand side), we have that
\eqn{
\label{Cauchy}
\left(\int_0^t \e^{-s}\expec[G(t-s)^{D^{\star}-1}]ds\right)^2\leq (1-\e^{-t})\int_0^t \e^{-s} \expec[G(t-s)^{D^{\star}-1}]^2ds.
}
Combining \eqref{V(1)-rec}, \eqref{edge_rec} and \eqref{Cauchy} yields that
$$
\prob\left(E+\widetilde V_a\wedge\widetilde V_b> t\right)\geq \e^{-t}+\frac{(G(t)-\e^{-t})^2}{1-\e^{-t}}\,=\,G(t)^2+\frac{\e^{-t}(G(t)-1)^2}{1-\e^{-t}}\geq \, G(t)^2,
$$
as desired.

Moving on to (b), recall that $V(k)=\sum_{j=1}^{\infty} E_j/S_j(k)$, where $S_j(k)=k+\sum_{i=1}^j (\widetilde B_i-1)$. Since $\widetilde B_i$ is in the domain of attraction of a stable law with exponent $\tau-2$, we have that $(S_{t k^{\tau-2}}(k)/k)_{t\geq 0} \convd (1+Q_t)_{t\geq 0}$, where $(Q_t)_{t\geq 0}$ is a stable subordinator with exponent $\tau-2$. Thus,
	$$
	k^{3-\tau} V(k)\convd \int_0^{\infty} 1/(1+Q_t)dt=:Y.
	$$
As for the expectation of the integral random variable $Y$, we use Fubini to write
	\eqan{
	\expec\Big[\int_0^{\infty} 1/(1+Q_t)dt\Big]
	&=\int_0^{\infty} \expec\big[1/(1+Q_t)\big]dt=\int_0^{\infty} \int_0^{\infty}\expec\big[\e^{-s(1+Q_t)}\big]dsdt\nn\\
	&=\int_0^{\infty} \int_0^{\infty}\e^{-s}\e^{-\sigma t s^{\tau-2}}dsdt
	=\frac{1}{\sigma}\int_0^{\infty}\e^{-s}s^{-(\tau-2)}ds<\infty\nn,
	}
where we have used that $\expec[\e^{-s Q_t}]=\e^{-\sigma t s^{\tau-2}}$ for some $\sigma>0$ and that $\tau-2\in (0,1)$.
\qed \medskip

Lemma \ref{lem-V(k)-asymp} allows us to prove \eqref{asympt-prob-win} in Theorem \ref{th:main-multiple}(c):

{\it Proof of \eqref{asympt-prob-win} in Theorem \ref{th:main-multiple}(c).} We note that
$V_{i,k}\stackrel{d}{=} V_i(A_{i,k})$, where $A_{i,k}=\sum_{j=1}^k D_{i,j}$ and $(D_{i,j})_{i,j\geq 1}$ are i.i.d.\ random variables with the same distribution as $D$. When $k\rightarrow \infty$, we have that $A_{i,k}/k\convp \expec[D]$. As a result, $(\expec[D]k)^{3-\tau}V_{1,k}\convd Y_1$, while $(\expec[D]k)^{3-\tau}V_{2,\alpha k}\convd \alpha^{\tau-3}Y_2$, where $Y_1,Y_2$ are i.i.d.\ copies of $Y$ in \eqref{Y-def}. Hence,
	\eqan{
	\prob(V_{i,k}<\mu V_{2,\alpha k})
	&=\prob\Big((\expec[D]k)^{3-\tau}V_{1,k}<\mu (\expec[D]k)^{3-\tau}V_{2,\alpha k}\Big)
	\rightarrow \prob\big(Y_1<\mu \alpha^{\tau-3}Y_2\big)\nn.
	}
\qed
\medskip

Let $B_m'(t)\stackrel{d}{=}B'_m|\,T'_m=t$. The next lemma shows that $(B'_m(t))_{m\geq 1}$ is stochastically dominated by an i.i.d.\ sequence whose mean is strictly smaller than 1 for large $t$. It also shows that $T'_m\to\infty$ almost surely. It is the key ingredient in the proof of Theorem \ref{th:main}:

\begin{lemma}[Asymptotic behavior of $B_m'(t)$ and $T'_m$]\label{le:B'm_dom}\leavevmode
\begin{itemize}
\item[\rm{(a)}] For each fixed $t>0$, the sequence $(B'_m(t))_{m\geq 1}$ is stochastically dominated by an i.i.d.\ sequence $(\bar{B}_m(t))_{m\geq 1}$. Furthermore, $\expec[\bar{B}_m(t)]$ is finite for each fixed $t>0$ and $\expec[\bar{B}_m(t)]\to \prob(D^{\star}=2)$ as $t\to\infty$.
\item[\rm{(b)}] Almost surely $T'_m\to\infty$ as $m\to\infty$.
\end{itemize}
\end{lemma}

\proof Recalling the definition of $B'_m|\,T'_m=t$ and using Lemma \ref{lem-V(k)-asymp}(a), we obtain for $k\geq 1$ that
	\eqn{
	\prob(B'_m(t)=k) 
	\leq \prob(D^\star=k+1)\prob(\muwin V(1)>t)^{k-1}.
	}
	
Denote $\bar{p}_k(t)=\prob(D^\star=k+1)\prob(\muwin V(1)>t)^{k-1}$, and let $\bar{B}_m(t)$ be defined by
$$
\prob(\bar{B}_m(t)=k)=\left\{
\begin{array}{ll}
\bar{p}_k(t) & \mbox{for }k\geq 1;\\
1-\sum_{k\geq 1} \bar{p}_k(t) & \mbox{for }k=0.
  \end{array}
            \right.
$$
For any fixed $t>0$, we have that $\bar{p}_k(t)\to 0$ exponentially in $k$. It follows that $\bar{B}_m(t)$ has all moments so that, in particular, its mean is finite. Furthermore, $\bar{p}_1(t)=\prob(D^\star=2)$ for any $t>0$, while $\bar{p}_k(t)\to 0$ as $t\to\infty$ for each $k\geq 2$. Hence $\expec[\bar{B}_m(t)]\rightarrow \prob(D^\star=2)$.

As for (b), recall the construction of the processes $(S'_m)_{m\geq 0}$ and $(T'_m)_{m\geq 0}$, with $S'_m=S'_0+\sum_{i=1}^mB'_i$ and $T'_m=\sum_{i=1}^mE'_i/S'_{i-1}$ for $m\geq 1$, where $B_m'=B_m(T_{m}')$. Note that $B'_i(t)$ is stochastically bounded by an i.i.d.\ sequence that is decreasing in $t$ having finite mean for all $t>0$, and that $t=T_m'\geq T_1'>0$ a.s. This implies that, conditionally on $T_1'=t_1'>0$, $S'_m$ grows at most linearly in $m$ and, as a consequence, $T'_m\to\infty$ a.s.
\qed\medskip

With this result at hand we are finally ready to prove Theorem \ref{th:main}(b):

\noindent \emph{Proof of Theorem \ref{th:main}(b).}
Recall the construction of the process $(S_m')_{m\geq 0}$ in the recursion (\ref{eq:Sm_rec}), and recall that $\NLna$ denotes the number of vertices infected by the losing type at time $\Totan$. Denote the total number of vertices infected by type $\mathcal{L}_n$ {\em after} time $\Totan$ by $\NLnaa$. We can identify this as
	$$
	\NLnaa=\#\{m\colon B_m'(n)\geq 1\}.
	$$
Indeed, each time when a new vertex is found that is not infected by type $\mathcal{W}_n$, by assumption (A1), the degree of the vertex is at least 2, so that $B_m'(n)\geq 1$. Thus, the number of vertices found is equal to the number of $m$ for which $B_m'(n)\geq 1$.

Recall that the total asymptotic number of losing type vertices is denoted by $\Nlos$. This number can now be expressed as
    	\eqn{
	\label{Nlos-tot-def}
	\Nlos=\Nlos^*+\Nlaa,
	}
where $\Nlos^*$ is defined in Lemma \ref{le:tt_explosion} and $\NLnaa\convd \Nlaa :=\#\{m\colon B_m'\geq 1\}$, where the weak convergence follows from Lemma \ref{lem-expl-losing}. Further, since the convergence in Lemma \ref{lem-expl-losing} is {\em conditional on} $\field_{\Totan},$ we also obtain the joint convergence
$$
(\NLna,\NLnaa)\convd (\Nlos^*,\Nlaa),
$$
which implies \eqref{Nlos-def}. To prove Theorem \ref{th:main}, it hence suffices to show that the random variable $\Nlaa$ is finite almost surely. This certainly follows when $H'=\max\{m\colon S_m'\geq 1\}$ is almost surely finite, which is what we shall prove below.

We argue by contradiction. Assume that $H'=\infty$. Then, $T_{m-1}'<\infty$ for every $m$.
Furthermore, $S'_m-S'_{m-1}=B'_m-1$ where, by Lemma \ref{le:B'm_dom}(a), the contribution $B'_m$ is stochastically dominated by $\bar{B}_m(T'_m)$, with $\expec[\bar{B}_m(t)]\to \prob(D^\star=2)$ as $t\to\infty$. Pick $k$ large so that $\expec[\bar{B}_m(T'_k)\mid T'_k]<1$, which is possible since $\prob(D^\star=2)<1$ and $T'_k\to\infty$ a.s.\ by Lemma \ref{le:B'm_dom}(b). Then, conditionally on $T'_k$ and for $m>k$, we have that $S'_m-S'_k$ is stochastically dominated by $\sum_{i=k+1}^m(\bar{B}_i(T'_k)-1)$ -- a sum of i.i.d.\ variables with negative mean. It follows that $S'_m$ hits 0 in finite time, so that $H'<\infty$, which is a contradiction.
\qed\medskip

We finish by proving Proposition \ref{prop-degree-time}:

\noindent {\it Proof of Proposition \ref{prop-degree-time}.}
We start with the proof of \eqref{Ntk-conv}. Let $U$ be a randomly chosen vertex and write $\1_{\sss U}^{\sss(t,k)}$ for the indicator taking the value 1 when vertex $U$ has degree $k$ and is occupied by type $\mathcal{W}_n$ at time $\Totan+t$. Note that, with $G_n$ denoting the realization of the configuration model including its edge weights,
	\eqn{
	\label{Ntk-repr}
	\Ntk=\expec[\1_{\sss U}^{\sss(t,k)}\mid G_n].
	}
We will show that $\E[\1_{\sss U}^{\sss(t,k)}\mid G_n]\convp \PP(\muwin V(k)\leq t)\PP(D=k)$ by aid of a conditional second moment method. Write $\SWG^{\sss(1,2)}(s)$ for the SWG at real time $s$ with exploration under competition. We perform the analysis conditionally on $\SWG^{\sss(1,2)}\left(T^{\sss (\sss\mathcal{W}_n)}_{\sss R_{a_n}}\right):=\Psi_n$, that is, the exploration graph under competition observed at the time when type $\mathcal{W}_n$ reaches size $a_n$ with one-type exploration, see below for further details on the structure of this graph. To apply the conditional second moment method, first note that
	\eqan{
	\expec[\Ntk \mid \Psi_n]&=\PP(\1_{\sss U}^{\sss(t,k)}=1 \mid \Psi_n)\nn\\
	&=\PP(\mbox{$U$ is type $\mathcal{W}_n$ infected at time $\Totan+t$}
	\mid \Psi_n,
	D_{\sss U}=k)\PP(D_{\sss U}=k),\nn
	}
and
	\eqan{
	\expec[(\Ntk)^2 \mid \Psi_n]
	&=\PP(\1_{\sss U_1}^{\sss(t,k)}=\1_{\sss U_2}^{\sss(t,k)}=1 \mid \Psi_n)\nn\\
	&=\PP(\mbox{$U_1,U_2$ are type $\mathcal{W}_n$ infected at time $\Totan+t$}\mid
	 \Psi_n, D_{\sss U_1}=D_{\sss U_2}=k)\nn\\
	&\qquad\times \PP(D_{\sss U_1}=k)\PP(D_{\sss U_2}=k),\nn
	}
where we use that the event $\{D_{\sss U_i}=k\}$ is independent of $\Psi_n$. Therefore, it suffices to show that the first factors in the above two right hand sides converge to $\PP(\muwin V(k)\leq t)$ and $\PP(\muwin V(k)\leq t)^2$, respectively. Indeed, in this case,
	$$
	\expec[\Ntk\mid \Psi_n]\convp \PP(\muwin V(k)\leq t) \prob(D=k),
	$$
while ${\mathrm{Var}}(\Ntk\mid \Psi_n)=\op(1)$, so that $\Ntk\convp \PP(\muwin V(k)\leq t) \prob(D=k)$, as required.

Assume that $\mathcal{W}_n=1$, so that type 1 wins whp. We can then construct $\Psi_n$ by first growing the one-type SWG from vertex 1 to size $a_n$. The time when this occurs is $T_{\sss a_n}^{\sss(1)}$, which converges in distribution to $V_1$. Then, we grow the one-type SWG from vertex 2 up to size $m_n=\sup\{m: \mu T^{\sss(2)}_{m}\leq T_{\sss a_n}^{\sss(1)}\}$. When $\mathcal{W}_n=1$, by Lemma \ref{le:tt_explosion}, the number $m_n$ of type 2 infected vertices at time $\Toan$ converges to an almost surely finite random variable. Furthermore, by Proposition \ref{prop:prel}(c), whp, $\SWG^{\sss(1)}_{a_n}$ and $\SWG^{\sss(2)}_{m_n}$ are disjoint. Hence, whp, $\Psi_n=\SWG^{\sss(1)}_{a_n}\cup\SWG^{\sss(2)}_{m_n}$.

Now recall that $X(1\leftrightarrow U)$ denotes the passage time between vertices 1 and $U$ in a one-type process with only type 1 infection. It follows from the analysis in \cite{RGSc}, summarized in Proposition \ref{prop:prel}, that $X(1\leftrightarrow U)$ converges in distribution to $V_1+V(k)$: As described above, we first grow $\SWG^{\sss(1)}_{a_n}$ and $\SWG^{\sss(2)}_{m_n}$. Then we grow the SWG from $U$ until it hits $\SWG^{\sss(1)}_{a_n}\cup \SWG^{\sss(2)}_{m_n}$. This occurs when the SWG from vertex $U$ has size $C_n$ such that $C_n/a_n$ converges in distribution to a proper random variable. Further, the time it takes to reach this size converges in distribution to $V(k)$ -- indeed, $V(k)$ describes the asymptotic explosion time for an exploration process started at a vertex with degree $k$. Hence,
	\begin{equation}\label{eq:onetypeconv}
	\PP(X(1\leftrightarrow U)\leq T_{\sss a_n}^{\sss(1)}+t\mid \Psi_n, D_{\sss U}=k)
	\convp \PP(V(k)\leq t).
	\end{equation}
In a similar way, we conclude that $\PP(X(1\leftrightarrow U_1),X(1\leftrightarrow U_2)\leq T_{\sss a_n}^{\sss(1)}+t\mid \Psi_n, D_{\sss U_1}=D_{\sss U_2}=k)\convp \PP(V(k)\leq t)^2.$
We need to show that the presence of type 2 infection started from vertex 2 does not affect this convergence result when $\mathcal{W}_n=1$.

Recall that $\SWG^{\sss (u)}(s)$ denotes the one-type SWG from vertex $u$ at time $s$. Also, let $\vep_n\searrow 0$ be as in Lemmas \ref{le:3}-\ref{le:4}. As pointed out above, the number $m_n$ of type 2 infected vertices at time $\Toan$ converges to an almost surely finite random variable. Furthermore, the probability that any additional vertices become type 2 infected in the time interval $(\Toan,\Toan+\vep_n)$ converges to 0, since $\vep_n\searrow 0$. Hence, \whp $\SWG^{\sss (1)}(\Toan)\cap \mu\SWG^{\sss (2)}(\Toan+\vep_n)=\varnothing$, where the multiplication by $\mu$ indicates that the edge passage times are multiplied by $\mu$ when constructing the SWG from vertex 2. Finally, by Lemma \ref{le:3}, the type 1 infection has \whp occupied all vertices with degree larger than $(\log n)^\sigma$ by time $\Toan+\vep_n$.

Now consider the SWG from vertex $U$ of degree $k$, where \whp $U\not\in \mu\SWG^{\sss(2)}(\Toan+\vep_n)$. Without the presence of the type 2 infection, this will hit $\SWG^{\sss(1)}(\Toan)$ when it has reached size $C_n$, where $C_n/a_n$ converges in distribution to a proper random variable, and the time for this converges in distribution to $V(k)$. We claim that \whp it does not hit the type 2 infection before this happens. This follows from Lemma \ref{le:2}: Indeed, the passage time from any vertex in $\mu\SWG^{\sss(2)}(\Toan+\vep_n)$ to $U$, not using the vertices in $\Goodn$ -- these are already occupied by the type 1 infection at time $\Toan+\vep_n$ and hence not available for the spread of type 2 -- is \whp larger than $b_n$, where $b_n\to \infty$. Hence, the passage time from any type 2 vertex to $U$ is \whp larger than $2V(k)+b$ for any $b>0$. This means that \whp the type 2 infection does not reach any of the vertices along the minimal weight path between $\SWG^{\sss(1)}(\Toan+\vep_n)$ and $U$ before time $V(k)+\vep$. Indeed, if it would, then there would be a path between vertex 2 and $U$ that avoids $\Goodn$ and that has passage time less than $2V(k)+\vep$.

It follows that when $\mathcal{W}_n=1$, the passage time between vertex 1 and $U$ behaves asymptotically the same as in a one-type process with only type 1 infection. Similarly, when $\mathcal{W}_n=2$, the passage time between vertex 2 and $U$ behaves asymptotically the same as in a one-type process with only type 2 infection, which yields an analog of (\ref{eq:onetypeconv}) where $V(k)$ is replaced by $\mu V(k)$. Furthermore, by the arguments in the proof of Lemma \ref{le:tt_explosion}, $\PP(\mathcal{W}_n=1)\to \PP(V_1<\mu V_2)$. Equation \eqref{Ntk-conv} in the proposition is hence proved.

The proof of \eqref{Ltwin-conv} is similar. Indeed, instead of \eqref{Ntk-repr}, we now start from $\barLtwin=\expec[\1_{e}^{\sss(t)}\mid G_n]$, where $e$ is a uniform edge in the graph and $\1_{e}^{\sss(t)}$ denotes the probability that $e$ spreads the infection before time $t$. We then again use a conditional second moment, and note that 	
	\[
	\expec[\barLtwin \mid \SWG^{\sss(1,2)}_{a_n}]
	=\prob(e\text{ has spread the $\mathcal{W}_n$ infection by time }\Totan+t \mid \SWG^{\sss(1,2)}_{a_n}).
	\]
In this expectation, a uniform edge can be obtained by drawing a half-edge uniformly at random, and pairing it to a uniform other half-edge. Let $a$ and $b$ be the vertices at the two ends of $e$, and let $\widetilde V_a$ and $\widetilde V_b$ be the explosion times of the vertices $a$ and $b$, respectively, when the type 1 infection is not allowed to use the edge $e$. Then, $e$ has spread the infection by time $\Totan+t$ precisely when either the explosion time $\widetilde V_a$ plus the edge weight $E_e$ are at most $t$ (in which case, $a$ is first type 1 infected and then spreads the infection to vertex $b$), or the explosion time $\widetilde V_b$ plus the edge weight $E_e$ are at most $t$ (in which case, $b$ is first type 1 infected and then spreads the infection to vertex $a$). We conclude that $\expec[\barLtwin \mid \SWG^{\sss(1,2)}_{a_n}]\convp \PP\big(\muwin\big(E+\widetilde V_a\wedge\widetilde V_b\big)\leq t\big)$. The extensions to the second moment computations as well as the fact that the competition does not interfere with the spread of the winning type are the same as for $\Ntk$.
\qed

\paragraph{Acknowledgement.} We thank the anonymous referee for careful reading of the manuscript, that has helped us identifying serious flaws in earlier versions and lead to significant improvements of our arguments. The work of MD was supported in part by the Swedish Research Council (VR) and The Bank of Sweden Tercentenary Foundation. The work of RvdH is supported by the Netherlands Organisation for Scientific Research (NWO) through VICI grant 639.033.806 and the Gravitation {\sc Networks} grant 024.002.003.

\end{document}